%% file: alt-jmiv21-va2.tex
\RequirePackage{fix-cm}
\documentclass[twocolumn]{svjour3_MODIFIED}       
\smartqed  
\usepackage{graphicx}
\journalname{Journal of Mathematical Imaging and Vision}

\usepackage{pgfplots}
\usepackage{tikz}
\usepackage{multirow}
\usepackage{amsmath, amssymb}
\usepackage{bm}
\usepackage{makecell}
\usepackage{mathtools}
\usepackage{subcaption}
\usepackage{pgfplotstable}
\usepackage{float}
\usepackage{csquotes}
\usetikzlibrary{matrix,calc,shapes,arrows}

\input{flowchart_framework.tex}

\newcommand{\nor}[1]{||#1||_2}
\newcommand\Tstrut{\rule{0pt}{2.6ex}}         
\newcommand\Bstrut{\rule[-0.9ex]{0pt}{0pt}}   

\def\STANDARDRESNET{0}
\def\STABLERESNET{1}
\def\DUFORTFRANKEL{2}
\def\FSI{3}

\def\CHARBONNIER{1}
\def\PERONAMALIK{2}

\def\RELU{4}

\def\UNCOUPLED{1}

\def\bestHDpsnr{31.20}
\def\bestCHpsnr{36.25}
\def\bestPMpsnr{37.22}

\begin{document}
\title{Connections between Numerical Algorithms for PDEs and Neural Networks
\thanks{This 
work has received funding from the European Research Council (ERC) 
under the European Union's Horizon 2020 research and innovation programme 
(grant agreement no. 741215, ERC Advanced Grant INCOVID).}
}
\subtitle{}


\author{Tobias Alt \and
        Karl Schrader \and 
        Matthias Augustin \and
        Pascal Peter \and 
        Joachim Weickert}


\institute{All authors are with the
Mathematical Image Analysis Group,
Faculty of Mathematics and Computer Science,
Campus E1.7, Saarland University,
66041 Saarbr\"ucken, Germany.\\
\email{alt@mia.uni-saarland.de}
}

\date{Received: date / Accepted: date}

\maketitle

\begin{abstract}
We investigate numerous structural connections between numerical algorithms for 
partial differential equations (PDEs) and neural architectures. Our goal is to 
transfer the rich set of mathematical foundations from the world of PDEs to 
neural networks. Besides structural insights we provide concrete 
examples and experimental evaluations of the resulting architectures. 
Using the example of generalised nonlinear diffusion in 1D, we 
consider explicit schemes, acceleration strategies thereof, implicit schemes, 
and multigrid approaches. We connect these concepts to 
residual networks, recurrent neural networks, and U-net architectures. Our 
findings inspire a symmetric residual network design with provable stability 
guarantees and justify the effectiveness of skip connections in neural networks 
from a numerical perspective. Moreover, we present U-net 
architectures that implement multigrid techniques for 
learning efficient solutions of partial differential equation models, and 
motivate uncommon design choices such as trainable nonmonotone activation 
functions. Experimental evaluations show that the proposed architectures save 
half of the trainable parameters and can thus outperform standard ones with the 
same model complexity. Our considerations serve as a basis for explaining the 
success of popular neural architectures and provide a blueprint for developing 
new mathematically well-founded neural building blocks.
\end{abstract}

\section{Introduction} 
Partial differential equations (PDEs) have been a central component of many 
successful models for signal and image processing in the last three decades; 
for instance, the monographs~\cite{AK06,CS05a,We97} provide a good overview. 

PDE-based models are compact, transparent, and benefit from a rich set of 
mathematical foundations. As a consequence, they offer valuable theoretical 
guarantees such as stability and well-posedness. This makes PDE models and 
their numerical solution strategies easy to control, implement, and apply to a 
plethora of tasks. 

Convolutional neural networks (CNNs) and deep learning 
~\cite{GBC16,LBH15,LBBH98,Sch15a}, on the other hand, have revolutionized the 
field of image processing in recent years. Still, this success 
was mostly of empirical nature. Many modern CNN architectures do not provide 
solid mathematical foundations. They often suffer from undesirable side effects 
such as a sensitivity against adversarial attacks~\cite{GSS15}. 

More recently, researchers have started to analyse the behaviour and 
mathematical foundations of CNNs. One strategy is to interpret trained 
CNNs as approximations of continuous partial or ordinary differential equations 
(ODEs)~\cite{CRBD18,CP16,DSBP21}. In this interpretation, the trainable 
parameters specify the nonlinear dynamics of an evolution equation. 

This strategy can be challenging for several reasons. Firstly, finding a 
compact differential equation is hard, given that typical CNNs learn millions 
of parameters. Secondly, the reduction of a discrete CNN to a continuous 
differential equation involves ambiguous limit assumptions, as the same 
discrete model can approximate multiple evolution equations with 
different orders of consistency. Last but not least, this strategy is only 
analytic: It analyses existing networks rather than inspiring 
novel, well-founded building blocks.

We address these problems by pursuing the opposite direction: We translate 
successful concepts from the world of PDEs into neural components. This 
translation justifies neural architectures from a mathematical perspective and 
provides novel design criteria for well-founded networks. In addition, these 
networks are naturally more compact with less trainable parameters.

Our concepts of choice are numerical algorithms instead of continuous 
differential equations. Similar to untrained neural networks, numerical 
algorithms can be applied to a multitude of problems in a general purpose 
fashion. The model at hand is then specified by the differential equation,
which is approximated by training the neural network. Thus, we believe that the 
design principles of modern neural networks realise a small but powerful set of 
numerical strategies at their core.  

\subsection{Our Contributions}
We investigate what can be learned from translating numerical methods for PDEs 
into their neural counterparts. This inspires novel building blocks for 
designing mathematically well-founded neural networks. 

The present paper advances our conference publication~\cite{APWS21}, 
in which we translated explicit schemes, acceleration strategies~\cite{HOWR16}, 
implicit schemes, and linear multigrid approaches~\cite{Br77,BHM00} into their 
neural counterparts. In this extended version, we additionally investigate 
networks that realise Du Fort--Frankel schemes~\cite{DF53} as a representative 
for absolutely stable schemes which are still explicit. Moreover, we extend the 
translation of multigrid approaches into U-nets to the nonlinear setting by 
considering full approximation schemes (FAS). Last but not least, we supplement 
our findings with an experimental evaluation of the proposed network 
architectures for denoising and inpainting tasks. Therein, we demonstrate the 
effectiveness of the architectures together with nonmonotone activation 
functions in practice.

As a starting point of our considerations, we consider a 
generalised nonlinear diffusion equation. We restrict ourselves to the 1D 
setting for didactic purposes only, since all important concepts can already be 
translated in this simple setting. We show that an explicit discretisation of 
this diffusion model can be connected to residual networks (ResNets) 
\cite{HZRS16}, which are among the most popular network architectures to date. 
Their skip connection can be interpreted as the result of a temporal 
discretisation, which allows to connect them to explicit diffusion schemes. 

This connection inspires a novel ResNet architecture. It follows a symmetric 
structure which saves half the amount of network parameters and additionally 
allows to derive a theory for guaranteeing stability in the Euclidean norm.
Moreover, by identifying the diffusion flux with the 
activation function our translation motivates the use of nonmonotone activation 
functions, which are atypical in the CNN world. In a series of denoising 
experiments, we validate the effectiveness of such activation functions. We 
choose the denoising problem since it is the prototypical representative of a 
well-posed problem, where the result depends continuously on the input data. 

By considering acceleration strategies for explicit schemes and solution 
strategies for implicit schemes, we justify the effectiveness of skip 
connections in neural networks. We show that Du Fort--Frankel schemes 
\cite{DF53}, fast semi-iterative (FSI) schemes~\cite{HOWR16}, and fixed point 
iterations for implicit schemes motivate different architectural designs which 
all rely on skip connections as a foundation of their efficiency.

Finally, we consider the rich class of multigrid approaches~\cite{Br77,BHM00}. 
We show that a nonlinear full approximation scheme (FAS) can be cast in the 
form of the popular U-net~\cite{RFB15} architecture. We think that at their 
core, U-nets realise a multigrid strategy, and we support this claim by 
proposing a U-net which realises a full multigrid strategy for an inpainting 
task.

Our findings do not only inspire new design criteria for stable neural 
architectures and show that uncommon design choices can perform well in 
practice. They also provide structural insights into the success of popular CNN 
architectures from the perspective of numerical algorithms.

\subsection{Related Work}
In recent years, the connection between neural networks and the world of PDEs 
and variational methods has become an active area of research.

CNNs are used to learn PDEs from data~\cite{LLD19,RMMW20,RBPK17,Sch17}, 
or to solve them efficiently~\cite{CP20,EHJ20,RPK19}. Moreover, various 
model-based approaches have been augmented with trainable parameters to improve 
their performance~\cite{AW21,AH20,CP16,KEKP20,KKHP17}. 
Another line of research is concerned with the  expressive power of networks 
\cite{DDFH21,GKNV21,KPRS21,PN21,RT18,TG19} 
and their robustness properties~\cite{CAH19,GMM20,LWF21}. 

The concept of neural ordinary equations~\cite{CRBD18} as a continuous 
time extension of ResNets~\cite{HZRS16} has gained considerable attention. 
However, recent works~\cite{GKDC21,OKHT21} suggest that these architectures 
suffer from a strong dependency between the model and the numerical solver. 
This supports our motivation to regard the numerical solver as an inherent 
basis of a neural architecture. 

In contrast, our philosophy of translating numerical concepts into neural 
architectures is shared only by few 
works~\cite{BCEO21,LHL20,LZLD18,OPF19,ZCF19}. They 
motivate additional or modified skip connections based on 
numerical schemes for ODEs, such as Runge-Kutta methods or implicit Euler 
schemes. Our work provides additional motivations for such skip connections 
based on several numerical strategies for PDEs.

The stability of ResNets has been analysed in several works
\cite{CMHR18,HLTR19,HR17,RDF20,RH20,ZS20}. A common result is that 
ResNets with a symmetric filter structure can be shown to be stable in the 
Euclidean norm. We motivate this result from a novel viewpoint based on 
diffusion processes. In contrast to previous results, this unique starting 
point allows us to present our stability result independently of the 
monotonicity of the activation function, inspiring the use of nonmonotone and 
trainable activation functions. 

Nonmonotone activation functions are rarely found in standard CNNs, with some 
notable exceptions~\cite{CP16,GWMC13,OMLM18}. Recently, the so-called Swish 
activation~\cite{RZL17} and modifications thereof~\cite{Mi20,ZMWZ21} have been 
found to empirically boost the classification performance of CNNs.
While these activations are modifications of the ReLU activation which are 
nonmonotone around the zero position, the activations that arise from our 
diffusion interpretation are odd and nonmonotone. Such functions have 
been analysed before the advent of deep learning~\cite{DMNP93,MR94}, but have 
not found their way into current CNN architectures. Our experiments, however, 
suggest that these activations can be advantageous in practice. 

Multigrid ideas have been combined with CNNs already in the early years of 
neural network research~\cite{Ba97,BMS93}. Current works use inspiration from 
multigrid concepts to learn restriction and prolongation operators of multigrid 
solvers~\cite{GGKY19,KDO17}, to couple channels for parameter reduction 
\cite{EERT20} or to boost training performance~\cite{HRHJ18,GRSC20}.

However, to the best of our knowledge, the only architectures 
that consequently implement a trainable multigrid approach are presented by He 
and Xu~\cite{HX19} and Hartmann et al.~\cite{HLMR20}. However, both works do 
not draw any connections to the popular U-net architecture~\cite{RFB15}, 
whereas we directly link both concepts.

\subsection{Organisation of the Paper}
In Section~\ref{sec:diffusionreview}, we review nonlinear diffusion and 
residual networks. We connect both models in Section~\ref{sec:diffusion}
and analyse the implications in terms of stability and novel activation 
functions. Afterwards we motivate skip connections from different numerical 
algorithms in Section~\ref{sec:skip}. We review multigrid approaches and U-nets 
in Section~\ref{sec:multigridreview} before connecting both worlds in 
Section~\ref{sec:multigrid}. Finally, we experimentally evaluate the proposed 
architectures in Section~\ref{sec:exp} and present a discussion and our 
conclusions in Section~\ref{sec:conc}.

\section{Review: Diffusion and Residual Networks}\label{sec:diffusionreview}

In this section, we review generalised diffusion filters in 1D and residual 
networks as the basic models for our first translation. We restrict ourselves 
to the 1D setting only for didactic reasons, as already this simple setting 
allows to translate all necessary concepts.

\subsection{Generalised Nonlinear Diffusion}
We start by considering a generalised one-dimensional diffusion PDE of 
arbitrary high order. It produces signals 
$u(x, t): (a,b) \times [0, \infty) \rightarrow \mathbb{R}$ evolving over time 
from an initial signal $f(x)$ on a domain $(a,b) \subset \mathbb{R}$ according 
to
\begin{equation}\label{eq:nldiff}
  \partial_t u = -\mathcal{D}^* \! \left(g\!\left(|\mathcal{D} u|^2\right) 
                                         \mathcal{D} u\right),
\end{equation}
with reflecting (homogeneous Neumann) boundary conditions. We use a 
general differential operator 
\begin{equation}
\mathcal{D} = \sum_{m=0}^M \alpha_m\partial_x^m
\end{equation} 
and its adjoint 
\begin{equation}
\mathcal D^*=\sum_{m=0}^M (-1)^m \alpha_m\partial_x^m.
\end{equation}
The operators consist of weighted derivatives up to order $M$ with weights 
$\alpha_m$ of arbitrary sign, yielding a PDE of order $2M$. 

Choosing e.g. $M=1$ yields the second order PDE of 
Perona and Malik~\cite{PM90}, while $M=2$ leads to a one-dimensional version of 
the fourth order model of You and Kaveh~\cite{YK00}. 

The evolution simplifies the input signal $f$ over time. This process is 
mainly controlled by the scalar \emph{diffusivity} function $g(s^2)$. For 
example, the Perona--Malik diffusivity~\cite{PM90} 
\begin{equation}\label{eq:peronamalik}
g(s^2) = \frac{1}{1 + \frac{s^2}{\lambda^2}}
\end{equation}
preserves discontinuities which are larger than a contrast 
parameter $\lambda$.

The diffusion PDE \eqref{eq:nldiff} is the gradient flow which minimises the 
energy functional
\begin{equation}
E(u) = \int_a^b \Psi(|\mathcal{D} u|^2)\, dx,
\end{equation}
where the penaliser $\Psi$ can be linked to the diffusivity with $g=\Psi'$ 
\cite{SchW98}. The penaliser must be increasing, but not necessarily convex. In 
Section~\ref{sec:nonmonotone}, we show that this inspires
novel activation functions. Their discretisations are stable, despite arising 
from a nonconvex energy.

\subsection{Residual Networks}
Residual Networks (ResNets)~\cite{HZRS16} belong to the most popular CNN 
architectures to date. Their main contribution is the introduction of 
so-called \emph{skip connections} which facilitate training of very deep 
networks. 

A residual network consists of residual blocks. A single 
block computes an output signal $\bm u$ from an input $\bm f$ as
\begin{equation}\label{eq:resblock}
\bm u = \sigma_2\!\left(\bm f + \bm W_2 \,\sigma_1\!\left(\bm W_1 \bm f + \bm 
b_1\right) + \bm b_2\right).
\end{equation}
One first applies an inner convolution $\bm W_1$ with a bias vector $\bm b_1$ 
to the input signal and passes the result into an inner \emph{activation} 
function $\sigma_1$.

Typically, CNNs prescribe simple, monotone activation functions such as the 
rectified linear unit (ReLU) \cite{NH10} function 
\begin{equation}
\text{ReLU}(s) = \text{max}(0,s)
\end{equation}
which is a linear function truncated at $0$. 

Afterwards an outer convolution $\bm W_2$ with a bias vector $\bm b_2$ is 
applied to the output of the activation.

The result of this convolution 
is added back to the original signal $\bm f$. This \emph{skip connection} is 
the crucial novelty of ResNets over feed-forward networks. It is the key to 
efficiently train deep networks with large amounts of layers, without suffering 
from the vanishing gradient problem. This phenomenon appears when 
backpropagation gradients approach zero for very deep networks, bringing 
the training process to a halt~\cite{BSF94}.  

Lastly, one applies an outer activation function $\sigma_2$ to obtain the
output signal $\bm u$.

\section{From Diffusion to Symmetric Residual Networks}\label{sec:diffusion}
We are now in the position to show that explicit diffusion schemes realise a 
ResNet architecture with a symmetric filter structure. To this end, we rewrite 
the generalised nonlinear diffusion equation \eqref{eq:nldiff} with the help of 
the flux function 
\begin{equation}
  \Phi(s) = g(s^2) \, s
\end{equation}
as 
\begin{equation}
  \partial_t u = -\mathcal{D}^* \Phi\!\left(\mathcal{D} u\right).
\end{equation}

Now we discretise this equation by means of a standard finite difference 
scheme. To obtain discrete signals $\bm u, \bm f$, we sample the continuous 
signals $u,f$ with distance $h$. We discretise the temporal 
derivative by a forward difference 
with time step size $\tau$. The spatial derivative operator $\mathcal D$ is 
implemented by a convolution matrix $\bm K$. 
Consequently, the adjoint operator $\mathcal D^*$ is realised by a transposed 
convolution matrix $\bm K^\top$. The matrix transposition 
corresponds mirroring the corresponding discrete convolution kernel. 

This yields the discrete evolution equation 
\begin{equation}
  \frac{\bm u^{k+1} - \bm u^k}{\tau} = - \bm K^\top
                                        \bm \Phi \! \left(\bm K \bm u^k\right),
\end{equation}
where we indicate old and new time levels by superscripts $k$ and $k+1$, 
respectively. Solving this expression for the new signal $\bm u^{k+1}$ yields 
the explicit scheme 
\begin{equation}\label{eq:explicit}
  \bm u^{k+1} = \bm u^k - \tau \bm K^\top
                          \bm \Phi \! \left(\bm K \bm u^k\right).
\end{equation}

The explicit diffusion scheme \eqref{eq:explicit} is closely connected to a 
specific residual block. To this end, we consider a residual block with input 
$\bm u^k$, output $\bm u^{k+1}$, and without bias terms, which reads
\begin{equation}
  \bm u^{k+1} = \sigma_2 \! \left(
                              \bm u^k + \bm W_2 \, \sigma_1 
                              \! \left(\bm W_1 \bm u^k\right) 
                           \right).
\end{equation}
Now, we can directly identify the explicit scheme with a residual block as 
follows.

\begin{theorem}[Diffusion-inspired ResNets]
An explicit step \eqref{eq:explicit} of the generalised higher order diffusion 
scheme \eqref{eq:nldiff} can be expressed as a residual block 
\eqref{eq:resblock} by 
\begin{equation}\label{eq:translation}
  \sigma_1 = \tau \, \bm\Phi, \quad
  \sigma_2 = \textup{Id}, \quad
  \bm W_1 = \bm K, \quad
  \bm W_2 = -\bm K^\top,
\end{equation}
with the bias vectors $\bm{b}_1$, $\bm{b}_2$ set to $\bm 0$. 
\end{theorem}
We call a ResNet block of this form a \emph{diffusion 
block}. Figure~\ref{fig:diffusion_block} visualises such a 
block in the form of a graph. Nodes contain the current state of the signal, 
while edges describe the operations to proceed from one node to the next. 


\begin{figure}[t]
  \centering
  \resizebox{0.8\linewidth}{!}{\input{diffusion_block.tex}}
  \vspace{-2mm}
  \caption{Diffusion block for an explicit    
           diffusion step \eqref{eq:explicit} with flux function $\bm\Phi$, 
           time 
          step size $\tau$, and a discrete derivative operator $\bm K$. 
          \label{fig:diffusion_block}}
\end{figure}
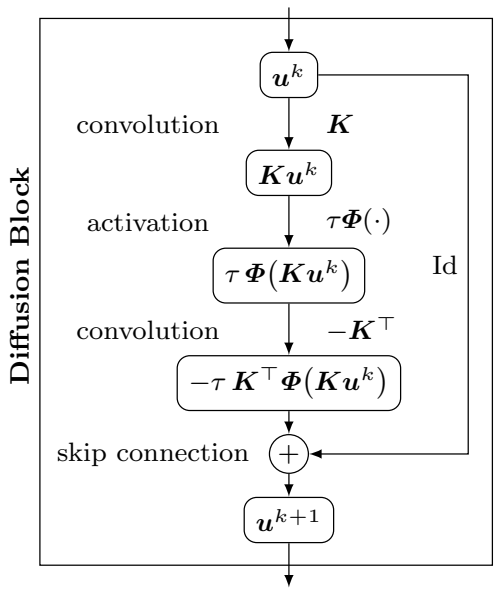


The connection between explicit diffusion schemes and ResNets yields three key 
structural insights:
\begin{enumerate}
\item The rescaled flux function $\tau \,\bm\Phi$ serves as the sole activation 
function $\sigma_1$. This motivates us to investigate popular diffusion flux 
functions in Section~\ref{sec:nonmonotone}.
They have not been considered as CNN activations so far. 
\item The skip connection naturally arises from the discretisation of 
the temporal derivative. This is one numerical justification for skip 
connections in neural networks. We investigate several other motivations of 
skip connections in Section~\ref{sec:skip}.
\item Lastly, the filters exhibit a negated symmetric filter structure $\bm 
W_2=-\bm W_1^\top$. This is a natural consequence of the gradient flow 
structure of the diffusion process, and leads to provable stability guarantees 
for ResNets with such a filter structure, as we show in the following.
\end{enumerate}

\subsection{Stability for Symmetric ResNets}
The structural connection between explicit schemes and ResNets allows us to
transfer classical results for stability~\cite{DWB09} and 
well-posedness~\cite{We97} of diffusion evolutions to a specific 
residual network architecture. 

To this end, we consider ResNets which chain diffusion blocks. Since we show 
that a key to stability of these networks is the symmetric filter structure, 
we refer to these architectures as \emph{symmetric residual networks
(SymResNets)}, following~\cite{RH20,ZS20}. 

For these networks, we prove Euclidean stability and well-posedness. Euclidean 
stability guarantees that the Euclidean norm of the signal is nonincreasing in 
each iteration, i.e. $\nor{\bm u^{k+1}} \leq \nor{\bm u^k}$. Well-posedness 
ensures that the network output is a continuous function of the input data.

\begin{theorem}[Euclidean Stability of Symmetric \\Residual Networks]
\label{theo:stab}
  Consider a symmetric residual network chaining any number of diffusion blocks 
  \eqref{eq:explicit} with convolutions represented by a convolution matrix 
  $\bm K$ and activation function $\tau \bm \Phi$. Moreover, assume that
  the activation function can be expressed as a diffusion flux function 
  $\Phi(s) = g(s^2) \, s$ and has a finite Lipschitz constant $L$. Then the 
  symmetric residual network is well-posed and stable in the Euclidean norm if 
  \begin{equation}\label{eq:bound}
   \tau \leq \frac{2}{L \nor{\bm K}^2}.
  \end{equation} 
  Here, $||\cdot||_2$ denotes the spectral norm which is induced by the 
  Euclidean norm. 
\end{theorem}
\begin{proof}
  The activation function $\sigma(s)$ can be expressed in terms of a diffusivity
  function by 
  \begin{equation}
  \sigma(s)=\tau\Phi(s) = \tau g(s^2) \, s.
  \end{equation}
  Thus, its application is equivalent to a rescaling with a diagonal matrix 
  $\bm G(\bm u^k)$ with $g((\bm K \bm u^k)^2_i)$ as $i$-th diagonal element. 
  Therefore, we can write \eqref{eq:explicit} as
  \begin{equation}
    \bm u^{k+1} = \left(\bm I - \tau \bm K^\top \bm G(\bm 
      u^k) \bm K \right) \bm u^k.
  \end{equation}
  At this point, well-posedness follows directly from the continuity of the 
  operator $\bm I - \tau \bm K^\top \bm G(\bm u^k) \bm K$, as  
  the diffusivity $g$ is assumed to be smooth~\cite{We97}.
    
  We now show that the time step size restriction \eqref{eq:bound} guarantees 
  that the eigenvalues of the operator always lie in the interval $\left[-1, 
  1\right]$. Then the explicit step \eqref{eq:explicit} constitutes a 
  contraction mapping which in turn guarantees Euclidean stability.
  
  As the spectral norm is sub-multiplicative, we can estimate the eigenvalues 
  of $\bm K^\top \bm G(\bm u^k) \bm K$ for each matrix separately.
  Since $g$ is nonnegative, the diagonal matrix~$\bm G$ is positive 
  semidefinite. The maximal eigenvalue of~$\bm G$ is then given by the 
  supremum of $g$, which can be bounded by the Lipschitz constant $L$ of $\Phi$:
  \begin{multline}
    L = \underset{s}{\text{sup}} \left|\Phi'(s)\right|
    = \underset{s}{\text{sup}} \left|g(s^2) + 2 s^2 g^\prime(s^2)\right|
    \\
    \geq \underset{s}{\text{sup}} \left|g(s^2)\right|.
  \end{multline}
  Consequently, the eigenvalues of $\bm K^\top \bm G(\bm u^k) \bm K$ lie in the 
  interval $\left[0, \tau L \nor{\bm K}^2\right]$. 
  
  Then the operator $\bm I - \tau \bm K^\top \bm G(\bm u^k) \bm K$ has 
  eigenvalues in $\left[1 - \tau L \nor{\bm K}^2, 1\right]$, and the condition 
  \begin{equation}
  1 - \tau L \nor{\bm K}^2 \geq -1 
  \end{equation}
  leads to the bound \eqref{eq:bound}.
  \qed
\end{proof}

Similar results have been obtained recently in~\cite{RDF20,RH20,ZS20}, albeit 
with alternative justifications. In Section~\ref{sec:nonmonotone} we show that 
our unique diffusion interpretation additionally suggests novel design concepts 
for CNNs such as nonmonotone activation functions.

\subsection{How General is Our Stability Result?}
While our focus on explicit diffusion schemes appears restrictive at 
first glance, our stability result is more general.

The fact that we use discrete differential operators as convolutions is no 
restriction, since any convolution matrix can be expressed as a weighted 
combination of discrete differential operators. Moreover, our proof does not 
even require a convolutional matrix structure. 

\begin{center}
\vspace{2mm}
\setlength{\fboxsep}{2mm}
\setlength{\fboxrule}{0.4mm}
\fbox{
\begin{minipage}{0.85\linewidth}\centering
{\bf A key requirement for stability is \\ the
symmetric structure $\bm W_2 = -\bm W_1^T$.} 
\end{minipage}
}
\vspace{2mm}
\end{center}

The symmetric convolution structure is an important structural difference to 
the original ResNet formulation~\cite{HZRS16}. It does not only yield a stable 
network, but also allows to reduce the amount of trainable 
parameters by 50\%, since inner and outer convolution share their weights.

Moreover, the requirement of using a flux function as an activation function 
can be relaxed. As we have shown, one only requires the diagonal matrix $\bm G$ 
to be positive semidefinite. While this is naturally fulfilled for a 
diffusion flux function, other activations also adhere to this constraint.
For example, the ReLU function multiplies positive arguments with $1$ and 
negative ones with $0$, yielding a binary positive semidefinite matrix $\bm 
G$. Thus, using the ReLU instead of a diffusion flux does not affect stability. 
This shows that diffusion algorithms inspire general, sufficient design 
criteria for stable networks.

In particular, we do not require any assumptions on the monotonicity of the 
activation function, in contrast to the results of Ruthotto and Haber 
\cite{RH20}. This motivates us to investigate typical diffusivities and their 
flux functions in Section~\ref{sec:nonmonotone}.

\subsection{Enforcing Stability in Practice}
While the stability criterion \eqref{eq:bound} can be computed on the fly 
already during the training process of the network, evaluating the spectral 
radius of the operator $\bm K$ is costly. To this end, we suggest a simple 
rescaling to turn the stability bound \eqref{eq:bound} into an a priori 
criterion. 

For a symmetric residual network with a single channel, one 
can directly use Gershgorin's circle theorem~\cite{Ge30} to bound the maximum 
absolute eigenvalue of $\bm K$. More precisely, the eigenvalues of $\bm K$ lie 
in the union of circles around the diagonal entries $k_{ii}$ with radii 
$r_i = \sum_{j\ne i} |k_{ij}|$ corresponding to the absolute sums of the 
off-diagonal values. Thus, the maximal absolute eigenvalue of $\bm K$ is 
bounded by the largest absolute row sum of $\bm K$. If we simply rescale both 
inner and outer convolutions by this sum, we can guarantee $\nor{\bm K}^2 
\leq 1$. Then the stability condition \eqref{eq:bound} transforms into $\tau 
\leq \frac{2}{L}$. Since the Lipschitz constant $L$ of the activation is known 
a priori, this simple rescaling allows to constrain the time step size to a 
fixed value, while not affecting the expressive power of the network. 

However, most networks in practice are not concerned with only a single 
channel. To this end, we extend our stability result to symmetric ResNets with 
multiple channels.

For a diffusion block operating on a signal with $C$ channels, the matrix $\bm 
K$ is a $C \times C$ block convolution matrix. As long as the transposed 
structure is realised, this is not problematic for the stability proof. 

An extension of Gershgorin's circle theorem to block matrices~\cite{Tr08} 
states that the eigenvalues of $\bm K$ lie in the union of circles which are 
centred around the eigenvalues of the diagonal blocks. The radii of the 
circles are given by the sum of the spectral norms of the off-diagonal blocks. 
If we rescale each block matrix as in the single channel case, we simply 
need to additionally divide the operator $\bm K$ by  
$\sqrt{C}$ to ensure that $\nor{\bm K}^2 \leq 1$. With this, we obtain the 
same a priori criterion as in the single channel case.

This strategy constitutes an instance of the popular weight normalisation 
technique~\cite{SK16}, and related spectral normalisations have shown to be 
successful for improving the performance and convergence speed of the training 
process~\cite{BRRS21,CP20a,GFPC21,ZBS20}.

\subsection{Nonmonotone Activation Functions}\label{sec:nonmonotone}

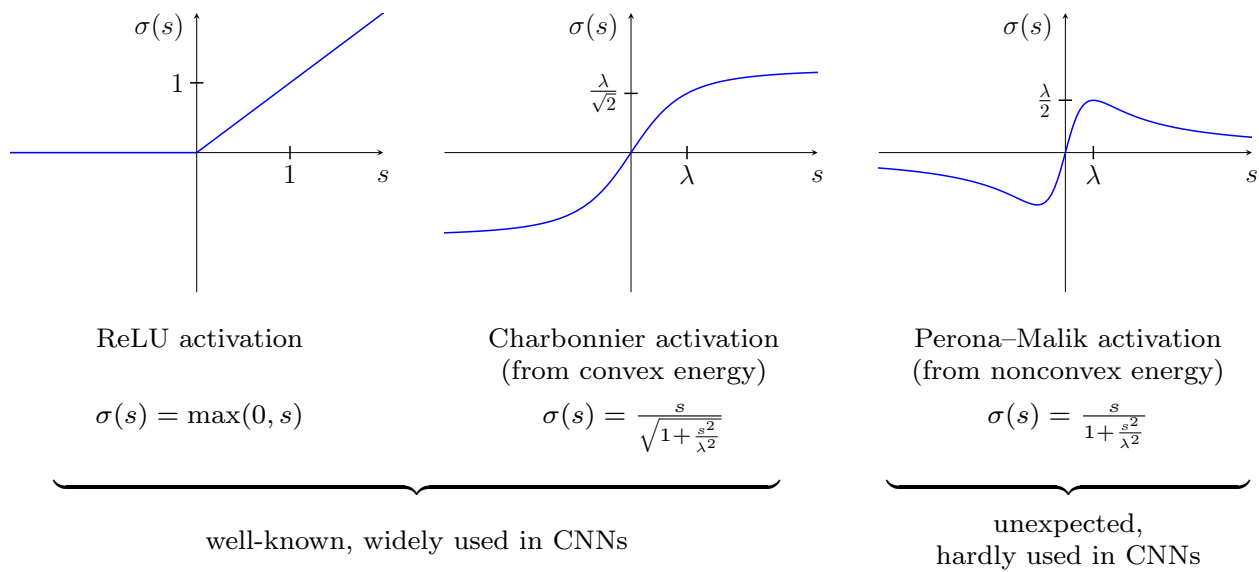
\begin{figure*}
  \centering
  \input{activations.tex}
  \caption{Visualisation of three activation functions. The ReLU (left) is a 
  standard activation function. The Charbonnier and Perona--Malik activations 
  represent two representatives of diffusion-inspired activation functions. 
  Both are dependent on a contrast parameter $\lambda$. The Charbonnier 
  activation is monotone, as it arises from a convex energy. On the contrary, 
  the Perona--Malik activation is nonmonotone, and its associated energy is 
  nonconvex. In this visualisation, we set $\tau=1$ for the sake of simplicity. 
  \label{fig:activations}}
\end{figure*}

Our connection between diffusivity $g(s^2)$ and activation function 
$\sigma(s) = \tau \, \Phi(s)$ with the diffusion flux $\Phi(s)$ refreshes an 
old idea of neural network design~\cite{DMNP93,MR94}. 

In Figure~\ref{fig:activations}, we present three activation functions: The 
ReLU activation~\cite{NH10}, along with two 
activations resulting from popular diffusivities. These activations are 
\emph{odd} functions, which is natural in the diffusion case, where the 
argument of the flux function consists of signal derivatives. It reflects the 
invariance axiom that signal negation and filtering are commutative.

The Charbonnier diffusivity~\cite{CBAB94}, which stems from a convex energy and 
can be seen as a rescaled regularised total variation 
diffusivity~\cite{ABCM98,ROF92}, yields a monotone activation function. 
Similarly shaped activation functions such as the hyperbolic tangent have been 
used in early neural networks, before being superseded by ReLU activations. 

The rational Perona--Malik diffusivity \eqref{eq:peronamalik}~\cite{PM90}, 
however, results in a nonmonotone activation function. The associated energy 
functional is nonconvex. Nevertheless, discretisations of diffusion processes 
using nonmonotone flux functions can be shown to be well-posed, 
despite acting contrast-enhancing~\cite{WB97}. For more activation functions 
inspired by diffusivities, we refer to~\cite{AWP20}. 

The concept of nonmonotone activation functions is unusual in the CNN world. 
Although there have been a few early proposals in the neural network literature 
arguing in favour of nonmonotone activations~\cite{DMNP93,MR94}, they are 
rarely used in modern CNNs. In practice, they often fix the activation to 
simple, monotone functions such as the rectified linear unit (ReLU).
From a PDE perspective, this appears restrictive. The diffusion interpretation 
suggests that activation functions should be learned in the same manner as 
convolution weights and biases.

In practice, this hardly happens apart from a 
few notable exceptions such as~\cite{CP16,GWMC13,OMLM18}. Recently, activation 
functions which are slightly nonmonotone variants of the ReLU proved successful 
for image classification tasks ~\cite{Mi20,RZL17,ZMWZ21}. As nonmonotone flux 
functions outperform monotone ones e.g. for diffusion-based denoising 
\cite{PM90}, it appears promising to incorporate them into CNNs. 

Our translation of explicit schemes is an example of a simple, direct 
correspondence which in turn allows for multiple novel insights. In the 
following, we explore variants of explicit schemes as well as implicit schemes 
which inspire changes to the skip connections of the symmetric ResNets, leading 
to more efficient architectures.

\section{The Value of Skip Connections}\label{sec:skip}
So far, we have seen that the temporal discretisation of an explicit scheme 
naturally leads to skip connections. This, however, is just one of the many 
justifications for their use. Since their proposal, skip 
connections~\cite{HZRS16} have been adapted into numerically inspired networks 
in many different forms; see e.g.~\cite{HLMW17,LHL20,LZLD18,OPF19,ZCF19}. 

This motivates us to explore several other numerical algorithms which justify  
several types of skip connections from a numerical perspective. We explore 
unconditionally stable schemes, acceleration strategies for explicit 
schemes, and fixed point iterations for implicit schemes.

\subsection{Du~Fort--Frankel Schemes}
While the classical explicit scheme \eqref{eq:explicit} is only conditionally 
stable, there exist absolutely stable schemes which are still explicit. These 
schemes are not that popular in practice since they trade unconditional 
stability for conditional consistency. However, we will see that this is not 
problematic from the perspective of learning.

Du Fort and Frankel~\cite{DF53} propose to change the temporal discretisation 
of the explicit scheme \eqref{eq:explicit} to a central difference and 
introduce a stabilisation term on the right hand side, corresponding to an 
approximation of $\partial_{tt} u$. A Du Fort--Frankel scheme for the 
generalised diffusion \eqref{eq:nldiff} evolution can be written as 
\begin{multline}
  \frac{\bm u^{k+1} - \bm u^{k-1}}{2 \tau} = - \bm K^\top 
  \bm\Phi\!\left(\bm K \bm u^k\right) 
  \\- \alpha \left(\bm u^{k+1} - 2 \bm u^k + \bm u^{k-1}\right),
\end{multline}
where a positive constant $\alpha$ controls the influence of the stabilisation 
term. 

Solving this scheme for $\bm u^{k+1}$ yields
\begin{multline}\label{eq:dufort}
  \bm u^{k+1} 
  = \frac{4\tau\alpha}{1 +2\tau \alpha}\left( \bm u^k- \frac{1}{2\alpha}\bm 
  K^\top \bm\Phi\!\left(\bm K \bm u^k\right) \right)
 \\ + \frac{1 - 2\tau \alpha}{1 + 2\tau \alpha} \bm u^{k-1}.
\end{multline}
For $\tau \alpha = \frac{1}{2}$, one obtains the explicit scheme 
\eqref{eq:explicit}.

The scheme involves the signals $\bm u^k$ and $\bm u^{k-1}$ at the current and 
the previous time level. The first term is nothing else than a rescaled 
diffusion block, where $\frac{1}{2 \alpha}$ takes the role of the original time 
step size. Since the scalar factors $\frac{4\tau\alpha}{1 +2\tau \alpha}$ and 
$\frac{1 - 2\tau \alpha}{1 + 2\tau \alpha}$ add up to $1$, this is simply an 
extrapolation of the result of an explicit step based on the signal at time 
level $k-1$.

If $\alpha$ is large enough, this scheme is unconditionally stable. Thus, one 
does not need to obey any stability condition, in contrast to the explicit 
case. 
Whereas classical proofs such as~\cite{GG76} consider only the linear case and 
typically work in the Fourier space, we are not aware of any proofs for the 
stability of nonlinear Du Fort--Frankel schemes. To this end, we prove 
stability of the nonlinear case in Appendix \ref{app:dufort}. 

However, this scheme is not unconditionally consistent. If the time step size 
$\tau$ is too large, the scheme \eqref{eq:dufort} approximates a different PDE 
\cite{DF53}, namely a nonlinear variant of the telegrapher's equation. Such 
PDEs have also been used in image processing; see e.g.~\cite{RZ09}.

In the trainable setting, the conditional consistency is not an issue, but can 
even present a chance. It allows the network to learn a more suitable PDE for 
the problem at hand. In our experiments we show that indeed, the unconditional 
stability of the Du Fort--Frankel scheme can help to achieve better results 
when only few residual blocks are available.

This scheme can be realised with a small change in the original diffusion block 
from Figure~\ref{fig:diffusion_block} by adding an additional skip connection. 
The two skip connections are weighted by $\frac{4\tau\alpha}{1 +2\tau \alpha}$ 
and $\frac{1 - 2\tau \alpha}{1 + 2\tau \alpha}$, respectively.

\subsection{Fast Semi-iterative Schemes}
Another numerical scheme which also leads to the same concept from a different 
motivation is based on acceleration strategies for explicit schemes.
Hafner et al.~\cite{HOWR16} introduced fast semi-iterative (FSI) schemes to 
accelerate explicit schemes for diffusion processes.

In a similar manner as the Du Fort--Frankel schemes, FSI extrapolates the 
diffusion result at a fractional time step $k+\frac{\ell}{L}$ with the previous 
fractional time step $k+\frac{\ell-1}{L}$ and a weight $\alpha_\ell$.
For the explicit diffusion scheme \eqref{eq:explicit}, an FSI acceleration with 
cycle length $L$ reads
\begin{multline}\label{eq:fsi}
  \bm u^{k+\frac{\ell+1}{L}} = \alpha_\ell 
  \left(\bm u^{k+\frac{\ell}{L}} - \tau \bm K^\top \bm \Phi\!\left(\bm K 
  \bm u^{k+\frac{\ell}{L}}\right)\right)
  \\ + \left(1-\alpha_\ell\right)\bm u^{k+\frac{\ell-1}{L}}
\end{multline}
with fractional time steps $\ell=0,\dots,L\!-\!1$ and extrapolation weights 
$\alpha_\ell \coloneqq (4\ell+2)/(2\ell+3)$. 
One formally initialises with $\bm{u}^{k-\frac{1}{L}} \coloneqq \bm u^{k}$. 
 
The crucial difference to Du Fort--Frankel schemes is that FSI schemes use 
time-varying extrapolation coefficients instead of fixed ones. These 
coefficients are motivated by a box filter factorisation and allow a cycle to 
realise a super time step of size $\frac{L(L+1)}{3}\tau$. Thus, with one cycle 
involving $L$ steps, one reaches a super step size of $\mathcal{O}(L^2)$ rather 
than $\mathcal{O}(L)$. This explains its remarkable efficiency~\cite{HOWR16}.

Even though Du Fort--Frankel and FSI schemes have fundamentally different 
motivations, they lead to the same architectural changes, where additional 
weighted skip connections realise acceleration strategies. This is in line with 
observations in the CNN literature; see e.g.~\cite{LZLD18,OPF19}. We visualise 
this concept at the example of an FSI architecture in 
Figure~\ref{fig:fsi_block}.


\begin{figure}[t]
  \centering
  \resizebox{0.7\linewidth}{!}{\input{fsi_block.tex}}
  \vspace{-2mm}
  \caption{FSI block realising the acceleration of an explicit    
           diffusion step \eqref{eq:fsi} with time-varying extrapolation 
           parameters $\alpha_\ell$. A similar architecture with differently 
           weighted skip connections arises for a Du Fort--Frankel scheme 
           \eqref{eq:dufort}. \label{fig:fsi_block}}
\end{figure}
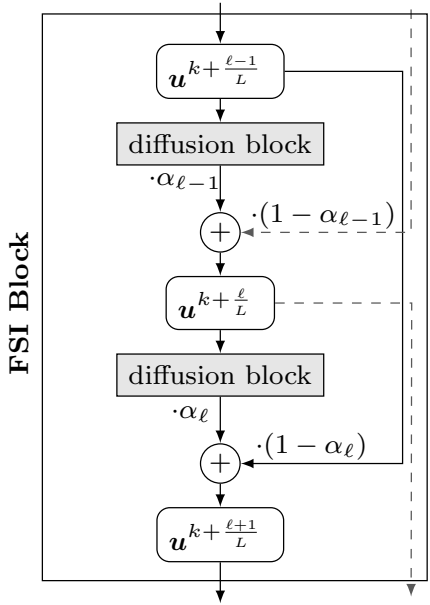


FSI and Du Fort--Frankel schemes are just two representatives of a large class 
of extrapolation strategies; see e.g.~\cite{Ne83,Po64,HS80}. The ongoing 
success of using momentum methods for training~\cite{RM86,SMDH13} and 
constructing~\cite{LZLD18,OPF19,ZCF19} neural networks warrants an extensive 
investigation of these strategies in both worlds. 

\subsection{Implicit Schemes}
So far, we have investigated variants of explicit schemes and their neural 
counterparts. However, implicit discretisations constitute another important 
solver class. We now show that such a discretisation of the generalised 
diffusion equation can be connected to a recurrent neural network 
(RNN). RNNs are classical neural network architectures; see e.g.~\cite{Ho82}. 
At the same time, this translation inspires yet another way of leveraging skip 
connections. 

A fully implicit discretisation of the diffusion equation \eqref{eq:nldiff} is 
given by
\begin{equation}
  \bm u^{k+1} = \bm 
  u^k - \tau \bm K^\top \bm \Phi\!\left(\bm K \bm u^{k+1}\right).
\end{equation}
The crucial difference as opposed to the explicit scheme lies in using the new 
signal $\bm u^{k+1}$ within the flux function $\bm \Phi$. This yields a 
nonlinear system of equations, which we solve by means of a fixed point 
iteration with a cycle of length $L$:
\begin{equation}
  \bm u^{k+\frac{\ell+1}{L}} = \bm u^k - \tau \bm K^\top \bm \Phi\! 
  \left(\bm K \bm u^{k+\frac{\ell}{L}}\right),
\end{equation}
where $\ell=0,\dots,L\!-\!1$, and where we assume that $\tau$ is sufficiently
small to yield a contraction mapping.

For $L=1$, we obtain the explicit scheme \eqref{eq:explicit} with its 
ResNet interpretation. For larger $L$, however, different skip connections  
arise. They connect the layer at time step $k$ with all subsequent 
layers at steps $k+\frac{\ell}{L}$ with $\ell=0,\dots,L\!-\!1$. 

There are two possible ways of interpreting this type of connection. One option 
is to regard this as a consequent extension of the extrapolation idea of FSI 
and Du Fort--Frankel schemes. Instead of only connecting a node to its two 
successors, the fixed point iteration above connects a node to $L$ of its 
successors. Similar ideas have been used in the popular DenseNet architecture 
\cite{HLMW17}, where each layer is connected to all subsequent ones. 

Another option is to interpret the repeated connection to $\bm u^k$ as a 
feedback loop, which in turn is closely connected to a recurrent neural network 
(RNN) architecture~\cite{Ho82}.

In their trainable nonlinear diffusion model, Chen and 
Pock~\cite{CP16} proposed a similar architecture. However, they explicitly 
supplement the diffusion process with an additional reaction term which results 
from the data term of the energy. Our feedback term is a pure numerical 
phenomenon of the fixed point solver.

We see that skip connections can implement a number of successful numerical 
concepts: forward difference approximations of the temporal derivative in 
explicit schemes, extrapolation steps to accelerate them e.g.~via FSI 
or Du Fort--Frankel schemes, and recurrent connections within fixed point 
solvers for implicit schemes.

\section{Review: Multigrid Solvers and U-nets}\label{sec:multigridreview}
Although the previously investigated numerical strategies and their neural 
counterparts can be efficient, they work on a single scale: The signal is 
considered at its original resolution at all points in time. However, using 
different signal resolutions in a clever combination can yield even higher 
efficiency.

This is the core idea of the large class of multigrid approaches 
\cite{Br77,BHM00,Hac85}. They belong to the most 
efficient numerical methods for PDE-related problems and have been 
successfully applied to various tasks such as image denoising~\cite{BC10a}, 
inpainting~\cite{MHWT12}, video compression~\cite{KSFR07}, and image 
sequence analysis~\cite{BWKS06}.

On the CNN side, architectures that work on multiple resolutions of the signal 
have become very successful. In particular, the shape of the popular U-net 
architecture~\cite{RFB15} suggests that there is a structural connection 
between multigrid and CNN concepts. 

By translating multigrid solvers into a U-Net architecture, we show that this 
is indeed the case, which serves as a basis for explaining the remarkable 
success of U-nets. Since both underlying concepts are not self-explanatory, we 
review them in the following before connecting them in the next section. 

\subsection{Multigrid Solvers for Nonlinear Systems}
Multigrid methods~\cite{Br77,BHM00,Hac85} are designed to accelerate 
the convergence speed of standard numerical solvers such as the Jacobi or the 
Gauß-Seidel method~\cite{Sa03}. These solvers attenuate high-frequent 
components of the residual error very quickly, while low-frequent error 
components are damped slowly. This causes a considerable drop in convergence 
speed after a few iterations.

Multigrid methods remedy this effect by transferring the low-frequent 
error to a coarser grid, transforming them into high-frequent components. This 
allows a coarse grid solver to attenuate them more efficiently. By correcting 
the fine grid approximation with coarse grid results, convergence speed 
can be significantly improved.

In the following, we review the so-called full approximation scheme (FAS) 
\cite{Br77} for a nonlinear system of equations. We consider a two-grid cycle 
as the basic building block of more complex multigrid solvers. 

We are interested in solving a nonlinear system of equations of the form 
\begin{equation}
 \bm A(\bm x) = \bm b,
\end{equation}
with a nonlinear operator $\bm A$ and a right hand side 
vector $\bm b$ for an unknown coefficient vector $\bm x$.  

The two-grid FAS involves two grids with different step-sizes: a fine 
grid of size $h$, and a coarse grid of size~$H~>~h$. We denote the respective 
grid by superscripts. The following six steps describe the two-grid FAS:
\begin{enumerate}
\item \textbf{Presmoothing Relaxation:} A standard solver is applied to the 
fine grid system $\bm A^h\!\left(\bm x^h\right) = \bm b^h$. Given an 
initialisation $\bm x_0^h$, it produces an approximation $\bm{\tilde x}^h$ to 
the solution with a reduced high frequency error.  

\item \textbf{Restriction:}
In order to approximate low-frequent components of the error more efficiently, 
one transfers the problem to the coarse grid with the help of a restriction 
operator $\bm R^{h\rightarrow H}$. One restricts both the 
residual $\bm r^h = \bm A^h\!\left(\bm x^h\right) - \bm b^h$ as well as the 
current approximation $\bm{\tilde x}^h$ to the coarse grid. One obtains two 
parts of the right hand side for the coarse grid problem: One part $\bm 
b^H = \bm R^{h\rightarrow H} \bm r^h$ which is used directly, and a second one 
which we denote by $\bm y^H =\bm R^{h\rightarrow H} \bm{\tilde x}^h$ serving as 
the argument for the nonlinear operator $\bm A^H$. The coarse grid problem 
then reads
\begin{equation}
  \bm A^H\!\left(\bm x^H\right) = \bm A^H\!\left(\bm y^H\right) + \bm b^H.
\end{equation}

If we express the desired solution $\bm x^H $ in terms of the error $\bm e^H$ 
by $\bm x^H=\bm y^H + \bm e^H$, then we see that this equation is solved for 
the full approximation rather than the error alone, in contrast to a linear 
multigrid scheme. Hence,this scheme is called the full approximation scheme. If 
the operator $\bm A$ is linear, FAS reduces to a linear multigrid scheme. 

A standard choice for $\bm R^{h\rightarrow H}$ is a simple averaging. However, 
finding suitable restriction operators is a difficult task, which motivates 
researchers to even learn such operators; see e.g.~\cite{GGKY19,KDO17}.

\item \textbf{Coarse Grid Computation:} Solving the coarse grid problem with a 
standard solver produces an error approximation $\bm{\tilde x}^H$. 

\item \textbf{Prolongation:} The approximation on the coarse grid needs to be 
transferred to the fine grid again.To this end, one applies a prolongation 
operator $\bm P^{H\rightarrow h}$. 

Since a coarse grid solution $\bm{\tilde x}^H$ is a full approximation, we need 
to compute the approximation to the error by $\bm{\tilde x}^H - \bm{y}^H$. This 
error approximation is then transferred to the fine grid via $\bm 
P^{H\rightarrow h}$.  

A standard choice for $\bm P^{H\rightarrow h}$ is a nearest neighbour 
interpolation, but as for the restriction operator, finding a good prolongation 
operator is not easy. 

\item \textbf{Correction:} The fine grid approximation $\bm{\tilde x}^h$ is 
corrected with the upsampled coarse grid error approximation 
$\bm P^{H\rightarrow h}\!\left(\bm{\tilde x}^H - \bm{y}^H\right)$ 
to produce a new approximation
\begin{equation}
\bm{\tilde x}^h_\text{new} = \bm{\tilde x}^h + \bm P^{H\rightarrow 
h}\!\left(\bm{\tilde x}^H - \bm{y}^H\right).
\end{equation}

\item \textbf{Postsmoothing Relaxation:} Finally, one applies another solver on 
the fine grid to smooth high frequent errors which have been introduced by the 
correction step. 
\end{enumerate}

The two-grid FAS will serve 
as the starting point for translating multigrid concepts into the a U-net 
formulation, in an extension to the linear connections from our conference 
publication~\cite{APWS21}.

\subsection{Review: U-nets}
U-nets~\cite{RFB15} are another popular neural architecture. They 
process information on multiple scales by repeatedly down- and upsampling the 
input data, interleaved with a series of convolutional network layers. This 
multiscale analysis makes them well-suited for medical image analysis tasks 
such as 
segmentation~\cite{DYLM17,RFB15}, but also for pose estimation~\cite{NYD16}
and shape generation~\cite{ESO18}.

A two-level U-net with fine grid size $h$ and coarse grid 
size $H$ has the following structure:
\begin{enumerate}
\item An input signal $\bm f^h$ is fed into a series of general convolutional 
layers which we denote by $\bm C^h_1\!\left(\cdot\right)$. The resulting output 
signal is denoted by $\bm{\tilde f^h} = \bm C^h_1\!\left(\bm f^h\right)$.  

Originally, these layers are assumed to be feed-for\-ward convolutional layers, 
but they can also be replaced by any other suitable layer type such as 
residual layers.  
 
\item The fine grid signal $\bm{\tilde f^h}$ is transferred to a coarser grid 
with a restriction operator $\bm R^{h\rightarrow H}$, yielding a coarse grid 
signal $\bm f^H = \bm R^{h\rightarrow H} \bm f$. 
  
\item On the coarse grid, another series of convolutional layers $\bm 
C^H\!\left(\cdot\right)$ is applied to the signal, yielding a modified coarse 
grid signal $\bm{\tilde f}^H = \bm C^H\!\left(\bm f^H\right)$.
 
\item The modified coarse grid signal is upsampled with a prolongation 
operator $\bm P^{H\rightarrow h}$.
 
\item With the help of a skip connection, the modified fine grid signal 
$\bm{\tilde f}^h$ and the upsampled coarse grid signal $\bm P^{H\rightarrow 
h}\bm{\tilde f}^H$ are added together. This produces a new fine grid signal 
$\bm{\tilde f}^h_\text{new}$. 
 
While the original U-net formulation~\cite{RFB15} suggests to concatenate both 
signals, other works such as~\cite{NYD16} simply add the signals. 
For our following discussion of connections between U-nets and multigrid 
schemes, we focus on the latter variant.

\item Lastly, another series of convolutional layers $\bm 
C_2^h\!\left(\cdot\right)$ is applied to the new fine grid signal, producing 
the final output signal $\hat{\bm f}^h$. 
\end{enumerate}

We visualise this architecture in Figure \ref{fig:multigrid_and_unet}(a).

\begin{figure*}[t]
\begin{subfigure}{\textwidth}
\centering
\input{unet.tex}
\caption{U-net architecture for an input $\bm f^h$.}
\vspace{5mm}
\end{subfigure}
\begin{subfigure}{\textwidth}
\centering
\input{fas.tex}
\caption{FAS two-grid cycle in the form of a U-net utilising four network 
channels. Besides the iteration variable $\bm x$, the network tracks variables 
$\bm b$ and $\bm y$ as linear and nonlinear parts of the system's right hand 
side, as well as the residual $\bm r$.}
\end{subfigure}
\caption{Architectures for a general U-net (a) and an FAS two-grid cycle (b).
\label{fig:multigrid_and_unet}}
\end{figure*}
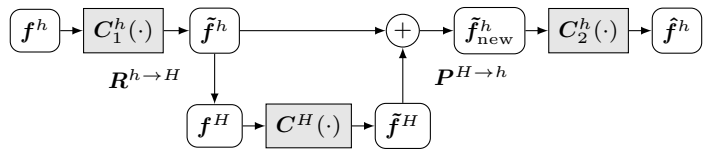
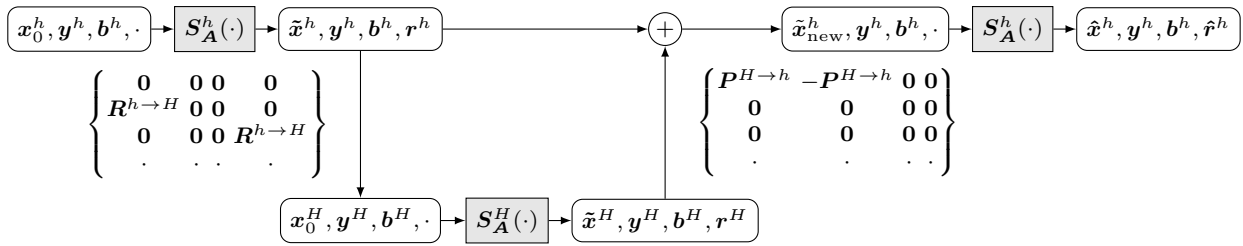

\section{From Multigrid to U-nets}\label{sec:multigrid}
Now we show how one can express FAS in terms of a U-net architecture. 
For our U-net, we use multiple network channels which carry the 
variables required by FAS. Even though not all variables are used at each 
point in the network, we keep the channel number consistent for the sake of 
simplicity. 

We track the FAS variables in dedicated channels only for didactic reasons, as 
a direct translation shows that a U-net architecture is 
sufficient for representing FAS. When practically implementing FAS, 
this overhead can be spared. 

Firstly, let us assume that we are given suitable solvers $\bm S_{\bm 
A}^h(\cdot), \bm S_{\bm A}^H(\cdot)$ for the nonlinear operators on the fine 
and coarse grid, respectively. To be able to use the two-grid cycle as a 
recursive building block, we assume that all solvers approximate solutions for 
nonlinear systems of the form $\bm A(\bm x) = \bm A(\bm y) + \bm b$, regardless 
of the grid.

To this end, we always keep track of the iteration variable $\bm x$, the 
nonlinear right hand side $\bm y$, and the linear right hand side $\bm b$. In 
addition, we track the residual~$\bm r$. By appropriately modifying these 
variables, we can ensure that the solvers always act on the desired system, 
despite having a common specification. 

\begin{enumerate}
\item \textbf{Presmoothing Relaxation:} The first instance of the fine grid 
solver obtains an initial iteration variable $\bm x_0^h$, which can be $\bm 0$ 
or a more sophisticated guess. Since the first solver is supposed to solve $\bm 
A^h\!\left(\bm x^h\right) = \bm b^h$, we simply set $\bm y^h=\bm 0$. In 
addition, we provide the linear right hand side $\bm b^h$. A residual is not 
needed as an input.

The solver produces a preliminary approximation $\bm{\tilde x}^h$, and passes  
the right hand side components $\bm y^h$ and $\bm b^h$ through without changes. 
It also computes the residual $\bm r^h = \bm b^h - \bm A^h\!\left(\bm{\tilde 
x}^h\right)$ as an additional output.

\item \textbf{Restriction:} As the downsampling is now explicitly concerned 
with four channels, the corresponding operator in our U-net is a $4\times4$ 
block matrix. We apply the multigrid restriction operator only to certain 
channels. 

The coarse grid initialisation $\bm x_0^H$ can be set to $\bm 0$, taking no 
information from the fine grid. The coarse, nonlinear right hand side $\bm y^H 
=\bm R^{h\rightarrow H} \bm{\tilde x}^h$ is given by the downsampled fine grid 
approximation. The corresponding linear right hand side $\bm b^H = \bm 
R^{h\rightarrow H} \bm r^h$ is the downsampled residual. 

In contrast to our linear correspondences in~\cite{APWS21}, the restriction 
step in FAS fits a U-net interpretation even better, since the 
approximation $\bm{\tilde x}^h$ itself is restricted, as is the case in the 
U-net.

\item \textbf{Coarse Grid Computation:} The coarse solver follows the same 
specification as the fine grid one. However, since $\bm y^H$ is not set to $\bm 
0$ at this point, the coarse solver actually solves the desired system $\bm 
A^H\!\left(\bm x^H\right) = \bm A^H\!\left(\bm y^H\right) + \bm b^H$. It 
produces a coarse approximation $\bm{\tilde x}^H$ and a residual $\bm r^H$,
while leaving the right hand side components unchanged. 

\item \textbf{Prolongation:} 
The upsampling step allows to prepare the coarse grid variables in such a way 
that the skip connection automatically performs the correct additions. 

The first row of the matrix operator ensures that we upsample the correction 
$\bm P^{H\rightarrow h}\bm{\tilde x}^H - \bm P^{H\rightarrow h}\bm{ y}^H$. 
Note that this is equivalent to the FAS formulation $\bm P^{H\rightarrow 
h}\!\left(\bm{\tilde x}^H - \bm{ y}^H\right)$ if the prolongation operator is 
linear. This is no limitation, however, since for the nonlinear case we can 
require the solvers to directly output the difference $\bm{\tilde x}^H - 
\bm{y}^H$.

The right hand side components $\bm y^H$ and $\bm b^H$ are not used in the 
upsampling, as the fine grid right hand side is supposed to be passed on. The 
same holds for the residual, as it is not relevant to the second fine grid 
solver. It is only needed in case one adds another coarser level to the cycle.

\item \textbf{Correction:} In the correction step, the fine approximation 
$\bm{\tilde x}^h$ is appropriately corrected, and the fine grid right hand side 
components $\bm y^h$ and $\bm b^h$ are forwarded. 

\item \textbf{Postsmoothing Relaxation:} Another instance of the fine grid 
solver solves the problem $\bm A^h\!\left(\bm x^h\right)= \bm b^h$. The 
nonlinear part $\bm y^h$ of the right hand side is still set to $\bm 0$, 
ensuring that the correct system is solved. 
\end{enumerate}

The resulting architecture is visualised in Figure 
\ref{fig:multigrid_and_unet}(b). This shows that U-nets share essential 
structural properties with multigrid methods. In particular, employing multiple 
image resolutions connected through pooling and upsampling operations, as well 
as horizontal skip connections which realise correction steps are the keys for 
the success of both methods. This leads us to believe that at their core, 
U-nets realise a sophisticated multigrid strategy. 

\subsection{V-Cycles, W-Cycles and Full Multigrid}
Our connections between 
two-grid FAS and U-nets are the basic building block for more advanced 
multigrid strategies.

So-called V-cycles arise from recursively stacking the two-grid FAS. Moreover, 
W-cycles can be built by concatenating several V-cycles. Optimising the depth 
and length of these cycles can lead to vast efficiency gains over direct 
solution strategies. 

On the CNN side, the corresponding concept of U-nets with more levels as well 
as concatenations thereof is successful in practice: Typical U-nets work on 
multiple resolutions~\cite{RFB15}, and so-called stacked hourglass models 
\cite{NYD16} arise by concatenating multiple V-cycle architectures.

A full multigrid (FMG) strategy solves a problem on multiple grids by 
successively concatenating V- and W-cycles, usually starting at the coarsest 
grid and progressing towards the finest one. In our experiments in 
Section~\ref{sec:mgexp} we will construct a trainable FMG model based on the 
two-grid FAS network to approximate the solution of an inpainting problem. This 
shows that our model reduction of the full U-net is successful in practice and 
inspires new design strategies for U-nets.  

\section{Experimental Evaluations}\label{sec:exp}
Let us now show that our findings are also of practical relevance.
Our experiments are divided into two parts. First, we evaluate the proposed 
symmetric ResNet architectures, along with their variations and nonmonotone 
activation functions for a denoising problem. 

In a second experiment, we make use of our connections between multigrid and 
U-nets to learn an efficient solver for diffusion-based sparse inpainting, 
based on a trainable FAS architecture.

\subsection{Symmetric ResNets and Nonmonotone Activations}
Since we motivate our network designs through numerical algorithms for a 
diffusion problem, we start with an elementary comparison on a denoising 
problem. We deliberately choose a denoising problem, since it is a prime 
example of a well-posed problem, for which the presented numerical schemes can 
be easily applied. 

In a second step, we refine the simple network structures to more 
and more complex ones, approaching the standard neural network design. 
This shows the extent to which our networks can compete with 
off-the-shelf ResNets.  

\subsubsection{Experimental Setup}
We compare symmetric ResNets and their Du Fort--Frankel and FSI extensions 
with the original ResNet architecture~\cite{HZRS16}. As activation 
functions we allow the ReLU~\cite{NH10}, Charbonnier 
\cite{CBAB94}, and Perona--Malik~\cite{PM90} activation functions. 

The original ResNets train two filter kernels per ResNet block, along with two 
bias terms. The symmetric ResNets on the other hand only train one filter 
kernel per block, without any bias terms. We only consider kernels of width 
three. For maximal transparency, we do not use any additional optimisation 
layers such as batch normalisation.

When using Charbonnier and Perona--Malik activations, we always train the 
corresponding contrast parameter $\lambda$. The Du Fort--Frankel networks also 
learn the extrapolation parameter $\alpha$, and the FSI networks train 
individual extrapolation parameters of each block. 

In addition, all models train their numerical parameters such as time step 
size and extrapolation parameters. We restrict the time step size $\tau$ to our 
stability condition \eqref{eq:bound} to obtain a stable symmetric ResNet model. 
In the case of the Du Fort--Frankel extension we restrict the extrapolation 
parameter $\alpha$ to the bound in Appendix \ref{app:dufort}, thus also 
yielding a stable scheme. For FSI, we restrict the extrapolation parameters 
$\alpha_{\ell}$ to the range $[0,2]$. This preserves the extrapolation 
character of the scheme. However, no stability theory is available in the case 
of learned extrapolation parameters.

We evaluate the networks on a synthetic dataset of 1D
signals which are piecewise affine, with jumps between the 
segments. This design highlights the ability of the different approaches to 
preserve signal discontinuities. The signals are of length $256$ and are 
composed of linear segments that span between $\frac{1}{10}$ and $\frac12$ of 
the signal length. Their values lie within the interval $[0,255]$.

Finally, we add Gaussian noise of standard deviation $\sigma=10$ to the 
signals, without clipping out of bounds values. This yields pairs of corrupted 
and ground truth signals. The training dataset contains $10000$ such pairs, and 
the test and validation datasets contain $1000$ pairs each. As a measure of 
denoising quality, we choose the peak-signal-to-noise ratio (PSNR), where 
higher values indicate better denoising performance. 

For a fair comparison, we train all network configurations in the same fashion. 
We use the Adam optimiser~\cite{KB14} with a learning rate of 
$0.001$ for at most $2000$ training epochs, and choose the average mean square 
error (MSE) over the training dataset as an optimisation objective. 

The filter weights are initialised according to a uniform random 
distribution with a range of $[-0.1, 0.1]$. The contrast parameters $\lambda$, 
the time step sizes $\tau$, and the extrapolation weights $\alpha$ are 
initialised with fixed values of $15$, $1.0$, and $1.0$, respectively. Out of 
several random initialisations, we choose the best performing one.

\begin{figure*}[h!]
  \centering 
  \resizebox{0.32\linewidth}{!}{\input{exp_atomic_relu.tex}}
  \resizebox{0.32\linewidth}{!}{\input{exp_atomic_charb.tex}}
  \resizebox{0.32\linewidth}{!}{\input{exp_atomic_pm.tex}}
  \caption{Denoising quality of network architectures with varying depth. We 
  use a single channel, and weights between all blocks are shared. Each plot is 
  concerned with a different activation function. Architectures with 
  Perona--Malik activation perform best, while the ReLU 
    activation is not suitable in this setting. Due to the tight network 
    constraints, the architectures reproduce the performance of classical 
    diffusion filters.\label{fig:atomic}}
\end{figure*}
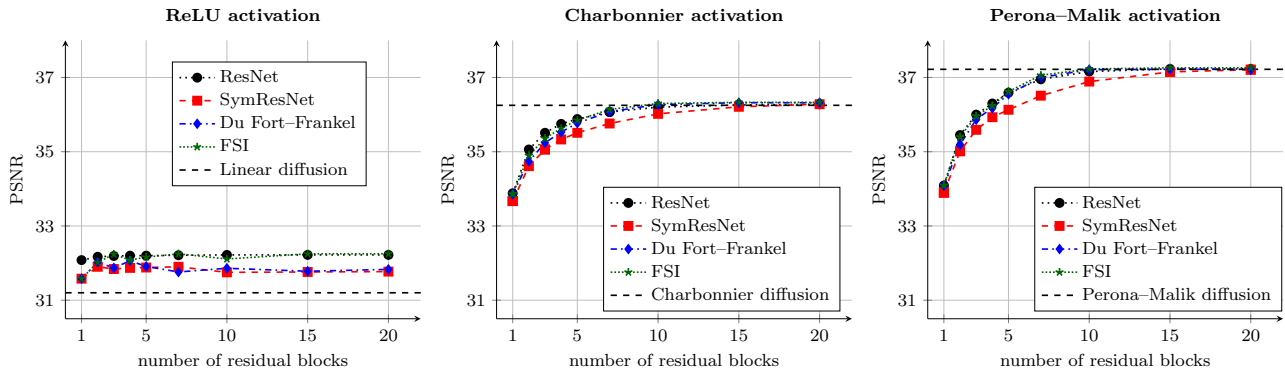
\subsubsection{Evaluation of Model Components}

We first evaluate the potential of the proposed network 
blocks on an individual level. To this end, we train the 
architectures for varying amounts of residual blocks. 
However, all blocks share their weights, and we also use only a single 
network channel. 

This configuration is closest to the interpretation of explicit schemes, and 
it allows us to investigate the approximation qualities of the different 
architectures within a tightly controlled frame. 

Figure~\ref{fig:atomic} presents the denoising quality of the architectures in 
dependence of the number of residual blocks. Each plot is concerned with a 
different activation function.

Firstly, we compare the different network architectures. As the symmetric 
ResNet is guaranteeing stability and uses less than half of the parameters of 
the standard ResNet, it performs slightly worse. This is not surprising, since 
there is a natural tradeoff between performance (high approximation
quality) and stability, as is well-known in the field of numerical analysis. 
Nevertheless, when enough blocks are provided, the symmetric ResNet catches up 
to the standard one.

The acceleration methods of Du Fort--Frankel and FSI outperform the 
symmetric ResNet and yield comparable performance to the standard ResNet. The 
trainable extrapolation parameters help both methods to achieve 
better quality especially when not enough residual blocks are provided to 
reach a sufficient denoising result. This is in full accordance with our 
expectations. When enough steps are provided and no extrapolation is required, 
both methods are on par with the standard and symmetric ResNets.

A side-by-side comparison yields insights into the performance of different 
activation functions. We observe that the performance of ReLU networks is only 
slightly better than classical linear diffusion~\cite{Ii62}. This shows that 
the ReLU is not suited for our denoising problem, regardless of the network 
architecture. After as few as three network blocks, the improvement of deeper 
networks is only marginal. 

In contrast, both the Charbonnier and the Perona--Malik activations are much 
more suitable. The nonmonotone Perona--Malik activation function yields the 
best denoising performance, as our diffusion interpretation suggests. 
When using the original ResNet with a diffusion-inspired 
activation function, tremendous performance gains in comparison to the ReLU 
activation can be achieved. This shows that in this experimental setting,  
the activation is the key to a good performance.

\subsubsection{Optimality of Diffusion Processes}
Interestingly, the standard ResNet with a diffusion activation
naturally learns a symmetric filter structure with biases close to $0$. 

For example, the ResNet with Perona--Malik activation, a grid size of $h=1$, 
and $20$ shared residual blocks learns an inner kernel $\bm k_1~=~(0.922, 
-0.917, 0.006)^\top$, an outer kernel $\bm k_2=(0.051, 0.437, -0.489)^\top$ and 
biases $b_1 = 1.6 \cdot 10^{-1}$ and $b_2 = 1.2 \cdot 10^{-5}$. 
If we factor out the time step size limit of $\tau=0.5$ for this setting from 
the outer kernel $\bm k_2$, we see that it transforms into $\bm{\tilde k}_2 = 
(0.102, 0.874, -0.978)^\top$. It becomes apparent that the kernels 
approximately fulfil the negated symmetric filter structure. 

Moreover, the kernels closely resemble rescaled standard forward and backward 
difference discretisations. This is surprising, as a kernel of width three 
allows to learn derivative operators of second order, but a first order 
operator appears to yield already optimal quality.

This shows that in this constrained setting, second order diffusion processes 
are an optimal model which is naturally learned by a residual network.

\begin{figure*}[h!]
  \centering 
  \resizebox{0.32\linewidth}{!}{\input{exp_dynamic_relu.tex}}
  \resizebox{0.32\linewidth}{!}{\input{exp_dynamic_charb.tex}}
  \resizebox{0.32\linewidth}{!}{\input{exp_dynamic_pm.tex}}
  \caption{Denoising quality of network architectures with varying depth and a 
  single channel. The parameters are smoothly changing between the residual 
  blocks. Each plot is concerned with a different activation function. 
  The proposed architectures can outperform the standard ResNet by saving a 
  large amount of parameters.
  \label{fig:dynamic}}
\end{figure*}
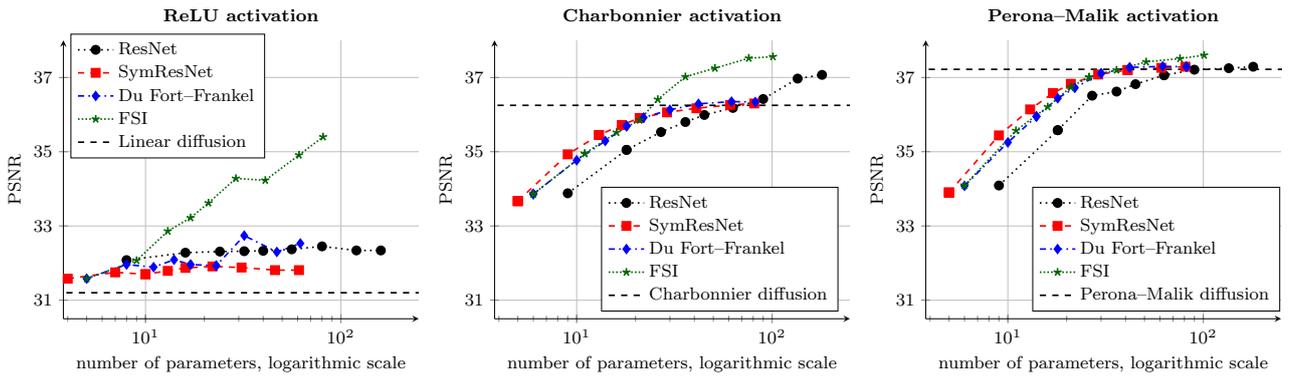

\subsubsection{Time Dynamic Case}
In a practical setting, the residual blocks typically do not share their 
weights, but train them independently. If the parameters evolve smoothly over 
the blocks, we can interpret this as an approximation of a time dynamic PDE 
model.

To investigate the performance of the proposed architectures in this setting, 
we train the parameters of each block individually, but enforce a certain 
smoothness between them. If the parameter vector of a block at time level $k$ 
is given by $\bm \theta^k$, we add a regulariser
\begin{equation}
  \beta \sum_{k=1}^{K} \tau \left(\frac{\bm \theta^{k} - \bm 
  \theta^{k-1}}{\tau}\right)^2
\end{equation}
to the loss function. Here, a smoothness parameter $\beta$ controls the amount 
of smoothness between the blocks, with higher values of $\beta$ leading to 
smoother evolutions. This expression approximates the continuous temporal 
regulariser
\begin{equation}
  \beta \int_{0}^{T} \left(\partial_t \bm \theta(t)\right)^2 \, dt,
\end{equation}
which enforces smoothness of the continuous parameter evolution~$\bm 
\theta(t)$. The regularisation ensures that the learned filters change 
smoothly throughout the layers. This is essential for the numerical scheme to 
be consistent with the continuous limit case where the step size~$\tau$ tends 
towards $0$; see also \cite{RH20}.

For the residual network, where no time step size $\tau$ is learned explicitly, 
we set the time step size to the inverse of the number of blocks. This requires 
to use a different smoothness parameter $\beta$. We tune the smoothness 
parameters for all architectures such that their parameters exhibit 
similarly smooth evolutions over time. Numerical parameters such as time step 
size and extrapolation parameters are not affected by the regulariser.

Figure~\ref{fig:dynamic} presents the performance of time dynamic architectures 
with a single channel. We use $\beta=5$ for the standard ResNets, and 
$\beta=10$ for all other architectures. In contrast to the previous 
comparisons, we now compare the denoising quality against the number of network 
parameters. This allows us to measure performance against model complexity. 

For the symmetric ResNet we observe the same behaviour as in 
the setting without a temporal dynamic. The overall best performance of the 
still on par with the respective classical diffusion process for the 
Charbonnier and Perona--Malik activation functions. 

However, the time dynamic allows this model to achieve better denoising quality 
for a fewer number of blocks. For example, the symmetric ResNet with 
Perona--Malik activation and seven residual blocks can achieve a denoising 
quality of $36.51$~dB if weights are shared, but already $37.08$~dB with a 
regularised temporal dynamic. Similar observations for classical diffusion 
processes with a time dynamic diffusivity function can be found 
in~\cite{GZS01}. Yet, this effectiveness comes at the price of additional 
parameters.

The Du Fort--Frankel and FSI networks allow for higher efficiency 
at the cost of more network parameters. Especially in the case of FSI, the 
trainable extrapolation parameters help to achieve significantly better 
performance when more residual blocks are provided. This is in accordance with 
observations in the literature \cite{LZLD18}.

The proposed architectures outperform the standard ResNet for the same 
amount of parameters when using Charbonnier and Perona--Malik activations. 
This shows that the model reduction to a symmetric convolution structure is 
indeed fruitful. Moreover, the ranking of activation functions remains the same 
in most cases. 

None of the architectures significantly outperform the classical 
Perona--Malik diffusion process. This supports our claim that these tightly 
constrained networks realise a numerical algorithm at their core. Different 
architectures can solve the problem with varying efficiency, but they converge 
towards the same result. This will only change when we allow for more 
flexibility within the network architecture, e.g. by utilising multiple network 
channels.

\begin{figure*}[h!]
  \centering 
  \resizebox{0.32\linewidth}{!}{\input{exp_channels_relu.tex}}
  \resizebox{0.32\linewidth}{!}{\input{exp_channels_charb.tex}}
  \resizebox{0.32\linewidth}{!}{\input{exp_channels_pm.tex}}
    \caption{Denoising quality of network architectures with varying depth and 
    $C=16$ network channels. Each plot is concerned with a different 
    activation function. The proposed architectures outperform the standard 
    ResNet for the same amount of parameters. In this setting, the acceleration 
    strategies are on par with the symmetric ResNet, and the margin between 
    activation functions becomes smaller.
    \label{fig:multichannel}}
\end{figure*}
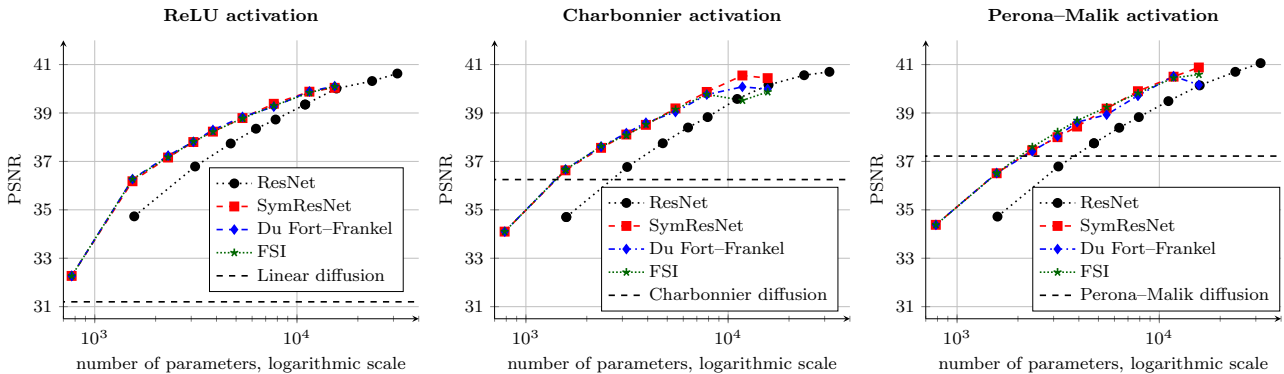

\subsubsection{Towards Larger Networks}
So far, we have only considered architectures with a single network channel. 
However, typical CNNs show their full potential when using multiple channels. 
In this case, the simplicity of the ReLU activation is compensated by a rich 
set of convolutions between channels.

We now extend our evaluations to architectures with multiple network channels. 
For simplicity, we leave the number of channels constant throughout the 
network. To this end, we copy the input signal into $C$ channels. The output 
signal is computed as the average over the individual channel results. 
In between, we employ the proposed symmetric residual blocks which now use 
$C\times C$ block convolution matrices with the appropriate stability 
constraints. These convolutions can be interpreted as ensembles of
differential operators which are applied to the signal. The channels are then 
activated individually, and convolved again with the adjoint counterparts of 
the differential operators.

 To allow for maximum flexibility and performance, we remove the 
temporal parameter regularisation in this experiment.

The denoising performance of the networks are visualised in 
Figure~\ref{fig:multichannel} for the case of $C=16$ channels.

This experiment is the first instance where all architectures significantly 
outperform their respective classical diffusion counterparts. The multi-channel 
architecture allows to approximate a more sophisticated denoising model, as 
information in the various channels is exchanged by means of multi-channel 
convolution operators.

All proposed architectures can still outperform the standard ResNet for the 
same amount of parameters. In this case, the extrapolation methods are on par 
with the symmetric ResNet. Interestingly, the ranking of activation functions 
remains the same, albeit with a much smaller margin. The symmetric ResNet with 
$20$ residual blocks yields PSNR values of $40.04$~dB, $40.44$~dB and 
$40.88$~dB for the ReLU, Charbonnier and Perona--Malik activations, 
respectively. We conclude that the more complex the network, the less the 
activation function matters for performance. On the contrary, this means that 
networks might be drastically reduced in size when trading network size for 
sophisticated activation function design.

\subsection{Learning a Multigrid Solver for Inpainting}
\label{sec:mgexp}
So far, it is not clear if our interpretation of the two-grid FAS is a 
reasonable model reduction of a full U-net. To prove that this interpretation 
is indeed of practical relevance, we show how it can be used to learn a 
multigrid solver for edge-enhancing diffusion inpainting of images 
\cite{We94e,WW06,GWWB08}.  

Diffusion-based inpainting aims to restore an image from a sparse set of known 
data points~\cite{GWWB08,WW06}. Diffusion processes allow for a high 
reconstruction quality even for extremely sparse known data, making them 
an interesting tool for image compression applications; see 
e.g.~\cite{GWWB08,SPME14}. 
Of particular interest is the edge-enhancing diffusion (EED) 
operator~\cite{We94e} as it allows to reconstruct discontinuous image data such 
as edges. 

As a proof of concept, we construct a network implementing a full multigrid 
structure. We replace the prescribed nonlinear solvers by trainable 
feed-forward layers that learn to approximate the PDE at hand. 

\subsubsection{Edge-enhancing Diffusion Inpainting}
The EED inpainting problem can be formulated as follows. 
Given a set of known image data $f: K\rightarrow \mathbb{R}$ on a subset $K$ of 
the image domain $\Omega \subset \mathbb{R}^2$, the goal is to compute a 
reconstruction $u$ as the solution of the PDE 
\begin{equation}\label{eq:inpaint}
  (1-c(\bm x)) \bm \nabla^\top\!\!\left(\bm D \bm \nabla u\right) - c(\bm x) 
  (u-f) = 0.
\end{equation}
Here, $c(\bm x)$ is a binary confidence function indicating whether the data at 
position $\bm x$ is known or not. The case $c(\bm x) = 1$ indicates known data, 
yielding $u=f$. The case $c(\bm x) = 0$ indicates that the data needs 
to be reconstructed by EED inpainting. Consequently, $c(\bm x)$ is the 
characteristic function of the inpainting mask $K$.

The EED operator $\bm \nabla^\top\!\!\left(\bm D \bm \nabla u\right)$ uses a 
diffusion tensor $\bm D = g(\bm \nabla u_\sigma \bm \nabla u_\sigma^\top)$ 
based on a Gaussian smoothed gradient $\bm \nabla_\sigma$ and a nonlinear 
diffusivity function $g$. It is a $2 \times 2$ positive semidefinite matrix, 
which is designed to propagate information along locally dominant structures 
\cite{We94e}. As a diffusivity function, we use the Charbonnier diffusivity 
\cite{CBAB94}, which relies on a contrast parameter $\lambda$.

\begin{figure}[t]
\centering
\input{fmg.tex}
\caption{Visualization of the full multigrid strategy which we employ in our 
experiments. Dashed horizontal lines denote the three grids, and grids become 
coarser from top to bottom. Each circle denotes an instance of a solver. 
\label{fig:fmg}}
\end{figure}
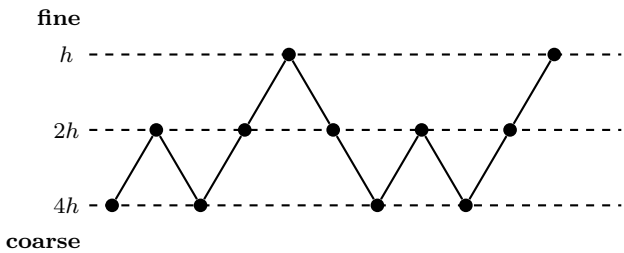

\begin{figure*}[h!]
\centering
\setlength{\tabcolsep}{2pt}
\begin{tabular}{ccc}
\includegraphics[width=0.21\textwidth]{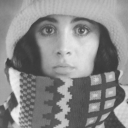} 
&\includegraphics[width=0.21\textwidth]{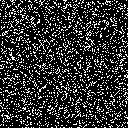} 
&\includegraphics[width=0.21\textwidth]{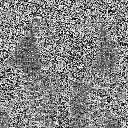}
\\[1mm]
original
& mask 
& \makecell{single grid network \\ residual $0.41$} 
\\ 
\includegraphics[width=0.21\textwidth]{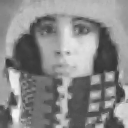}
&\includegraphics[width=0.21\textwidth]{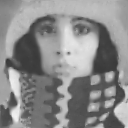} 
&\includegraphics[width=0.21\textwidth]{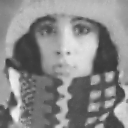} 
\\[1mm]
 \makecell{full multigrid network \\ residual $1.7 \times 10^{-2}$}
 & \makecell{standard U-net \\ residual $7.4\times10^{-3}$}
& \makecell{CG solver \\ residual $1.0 \times 10^{-8}$}
\end{tabular}
\caption{Reconstruction quality of a single grid network, a full multigrid 
network, a standard U-net, and a classical CG solver for the EED 
inpainting problem. All results use $\lambda=0.93, \sigma=0.97$ and the same 
random mask with $20\%$ density. Both the full multigrid network and the 
standard U-net approximate an EED inpainting result, while a single grid 
network fails to do so. 
\label{fig:inpainting}}
\end{figure*} 
\subsubsection{Experimental Setup}
Our trainable FMG architecture is designed as follows: Instead of prescribing 
nonlinear solvers on each grid, we employ a series of
convolutional layers with trainable weights. The remainder of the architecture 
is fixed: We set the restriction operators to a simple 
averaging over a $2\times 2$ pixel neighbourhood, and the prolongation operators
to nearest neighbour interpolation. 

Since FMG employs the same solver on each 
grid, we realise this idea also in our network by sharing the weights between 
all solvers for a specific grid. This drastically reduces the amount of 
parameters and incites that an iterated application of the solvers performs the 
correct computations.

Instead of training our network by minimizing a Euclidean loss between ground 
truth data and inpainting reconstructions, we use the absolute residual of a 
discretisation of the inpainting equation \eqref{eq:inpaint} as a loss 
function. This is closely related to the idea of deep energies~\cite{GFE21}, 
where one chooses a variational energy as a loss function. Since we do not 
have such an energy available for EED inpainting, we resort to minimising the 
absolute residual of the associated Euler--Lagrange equation, which is given by 
\eqref{eq:inpaint}. This guarantees that the trained architecture realises EED 
inpainting as efficiently as possible. To discretise the Euler--Lagrange 
equation, we employ the standard discretisation from~\cite{WWW13}.

Whereas classical solution methods for the inpainting problem specify the 
known data $u=f$ on $\Omega\backslash K$ by means of Dirichlet boundary 
conditions, we leave it to the network to reproduce also the known data. We 
have found that this leads to a better approximation quality.

The inpainting masks consist of randomly sampled pixels with a density 
$d$ as a percentage of the number of image pixels. Since the masks are also 
required on coarser grids to compute the residual within the FMG architecture, 
we downsample them by putting defining a coarse pixel as known, if at least one 
pixel in the $2 \times 2$ cell on the fine grid is known.

We train the architecture on a subset of $1000$ images of the ImageNet dataset 
\cite{RDSK15} with the Adam optimiser~\cite{KB14} with standard settings.

\subsubsection{Evaluation of the Full Multigrid Network}
We construct a full multigrid network using three grids 
of size $h, 2h$ and, $4h$. The order in which the problem is solved on 
different grids is given by $[4h, 2h, 4h, 2h, h, 2h,$ $4h, 2h, 4h, 2h, 
h]$. This is the simplest FMG strategy that can be employed in 
a setting involving three grids and serves as a proof-of-concept 
architecture. We visualise this strategy in Figure \ref{fig:fmg}. Thus, we 
employ $11$ solvers, each using $12$ feed-forward convolutional layers with 
$20$ channels and ReLU activations. Weights are shared for solvers on the same 
grid.

We train this network on an EED inpainting problem with random masks 
of $20\%$ density. The EED parameters $\lambda=0.93, \sigma=0.97$ have 
been optimised for inpainting quality with a simple grid search.

To show how the FMG network can benefit from the multigrid 
structure, we compare it against two networks with the same amount of 
parameters. One network solves the problem only on a single grid by using $25$ 
layers and $24$ channels. Moreover, we compare our FMG network to a standard 
U-net with addition with three scales, $17$ channels and $2$ layers per scale. 
All three models contain $1.2 \times 10^5$ trainable parameters.

Figure~\ref{fig:inpainting} shows the inpainting results for the networks at 
the example of the image \emph{trui}. Moreover, we present the true inpainting 
result obtained from a conjugate gradient (CG) solver for EED inpainting. 

In contrast to the single grid network, the FMG network and the U-net are 
able to approximate the EED inpainting result. The FMG and U-net results 
are visually comparable, while the residual of the U-net is slightly better. 
This does not only show that the multigrid structure is an adequate network 
design, but also that a standard U-net is not able to obtain a much better 
solution in this case. From the perspective of numerical algorithms, this is 
expected, since we know that multigrid methods are highly efficient for these 
problems. 

The advantage of the FMG network is that adding further solvers on the 
three scales will not inflate the number of parameters as these solvers are 
shared. This is in contrast to the U-net, where any addition increases the 
trainable parameter set. The architectural design of the FMG network suggests 
that U-nets should also be constructed in a similar way. In practice, already 
concatenating multiple U-nets~\cite{NYD16} is a successful idea. Instead of a 
single down- and upsampling pass, multiple alternating computations on 
different resolutions should be beneficial.

\begin{table*}[h!]
\begin{center}
\renewcommand{\arraystretch}{1.1}
\resizebox{0.9\linewidth}{!}{
\begin{tabular}{|l|l|l}
\cline{1-2} 
Numerical Concept 
& Neural Concept\Tstrut\Bstrut 
&
\\ \cline{1-2}  
  numerical algorithm
& neural network architecture\Tstrut
& \multirow{3}{*}{$\left.   
  \vphantom{\begin{tabular}{c}x\\x\\x\end{tabular}}\right\}$ 
  \textbf{interpretation}}
\\
evolution equation
& trained neural network
&
\\
specification of nonlinear dynamics
& training
&
\\[1mm]
explicit scheme
& residual network
& \multirow{3}{*}{$\left.   
  \vphantom{\begin{tabular}{c}x\\x\\x\end{tabular}}\right\}$ \textbf{coarse 
  connections}}
\\
implicit scheme
& recurrent network
&
\\
multigrid techniques
& U-net architectures
&
\\[1mm]
time level 
& residual block
& \multirow{5}{*}{$\left.   
  \vphantom{\begin{tabular}{c}x\\x\\x\\x\\x\end{tabular}}\right\}$  
  \textbf{detailed connections}}
\\
flux function 
& activation function 
&
\\
spatial derivative 
& convolution kernel 
&
\\ 
temporal derivative 
& skip connection
& 
\\
acceleration strategies 
& weighted skip connections\Bstrut\\\cline{1-2}  
\end{tabular}}
\end{center} 
\caption{Overview of the connections between numerical and neural 
concepts which we have encountered. \label{tab:connections}}
\end{table*}
\section{Discussion and Conclusions}\label{sec:conc}
We have shown that numerical algorithms for diffusion evolutions 
share structural connections to CNN architectures and inspire novel design 
concepts. 

Explicit diffusion schemes yield a specific form of residual networks with a 
symmetric filter structure, for which one can prove Euclidean stability. 
Moreover, this architecture saves half of the network parameters, and its 
stability constraint is easy to implement without affecting its performance. In 
addition, our connection suggests the use of a diffusion flux function as an 
activation, revitalising the idea of nonmonotone activation functions. We have 
shown that these activations perform well for a denoising task, even when using 
them within standard architectures in a plug-and-play fashion. 

By investigating accelerated explicit schemes and implicit 
schemes we have justified the effectiveness of skip connections in neural 
networks. They realise time discretisations in explicit schemes, extrapolation 
terms to increase their efficiency, and recurrent connections in implicit 
schemes with fixed point structure. In practice, the resulting architectures 
are particularly useful when training networks with small amounts of layers. 

Lastly, our connection between multigrid concepts and U-net architectures 
serves as a basis for explaining their success. We have shown that a U-net 
architecture is able to implement a full multigrid strategy, which allows to 
learn efficient solutions for PDEs which are typically hard to solve. By 
directly using the residual norm as a loss function, we can guarantee that the 
network approximates the PDE at hand. This suggests to extend the standard 
U-net architecture in a full multigrid fashion.

Our philosophy of identifying numerical concepts as core building blocks of 
neural architectures has proven to be fruitful. Our direct translation has 
yielded structural insights into popular neural networks and inspired 
well-founded neural building blocks with practical relevance. An overview of 
the detailed connections that we have encountered is presented in Table 
\ref{tab:connections}.

Our numerical perspective on neural networks differs from most viewpoints in 
the literature. However, it provides a blueprint for directly translating a 
plethora of numerical strategies into well-founded and practically relevant 
neural building components. We hope that this line of research leads to a 
closer connection of both worlds and to hybrid methods that unite the stability 
and efficiency of modern numerical algorithms with the performance of neural 
networks.

\appendix
\section{Stability Proof for Generalised Du~Fort--Frankel Schemes}
\label{app:dufort}
In this section we extend the stability proof for linear Du Fort--Frankel 
Schemes of~\cite{GG76} to the nonlinear setting. 

First we rewrite the scheme \eqref{eq:dufort} as a multi-step method, obtaining
\begin{equation}
 \label{eq:DF_multi}
 \begin{pmatrix}
  \bm{u}^{k+1} \\
  \bm{u}^{k}
 \end{pmatrix}
 =
 \begin{pmatrix}
 \frac{4 \tau \alpha}{1 + 2 \tau \alpha} \bm{I}
 -\frac{2 \tau }{1 + 2 \tau \alpha} 
   \bm{A}\!\left(\bm{u}^{k}\right)
  &
  \frac{1 - 2 \tau \alpha}{1 + 2 \tau \alpha}
  \bm{I} 
  \\[2mm]
  \bm{I}
  &
  \bm{0}
 \end{pmatrix}
 \begin{pmatrix}
  \bm{u}^{k} \\
  \bm{u}^{k-1}
 \end{pmatrix}.
\end{equation}
Here, we have abbreviated $\bm A(\bm u^k) = \bm K^\top \bm G(\bm u^k) \bm K$.

To analyse the stability of the Du Fort--Frankel scheme, we have to show that 
all eigenvalues of the matrix of the multistep method \eqref{eq:DF_multi}
have an absolute value less than or equal to one. Note that this matrix is not 
symmetric such that it might have complex eigenvalues.

Let us first define as short-hand notations
\begin{equation}
 \begin{split}
 \bm{Q} \coloneqq&\   
 \frac{4 \tau \alpha}{1 + 2 \tau \alpha} \bm{I}
 - \frac{2 \tau }{1 + 2 \tau \alpha} 
   \bm{A}\!\left(\bm{u}^{k}\right),
 \\[1ex]
 \bm{B}\coloneqq&\
 \begin{pmatrix}
  \bm{Q} & \frac{1 - 2 \tau \alpha}{1 + 2 \tau \alpha} \bm{I} \\[2mm]
  \bm{I} & \bm{0}
 \end{pmatrix}.
 \end{split}
\end{equation}

We start with the naive approach to compute eigenvalues of $\bm{B}$: 
$\mu \in \mathbb{C}$ is an eigenvalue of $\bm{B}$ if
\begin{equation}
 \label{eq:ev_B_01}
 \det\!\left(\bm{B} - \mu \bm{I}\right) 
 =
 \det\!\left(
  \begin{pmatrix}
   \bm{Q} - \mu \bm I & \frac{1 - 2 \tau \alpha}{1 + 2 \tau \alpha} \bm{I} 
   \\[2mm]
   \bm{I} & - \mu \bm I
  \end{pmatrix}
 \right) = 0.
\end{equation}
As $\bm B - \mu \bm I$ is a block matrix containing square blocks of the same 
shape where the lower two blocks commute, we have
\begin{equation}
\begin{aligned}
 \det\!\left(\bm B - \mu \bm{I}\right)
 &= \det\!\left(
      \left(\bm{Q} - \mu \bm{I}\right) (-\mu \bm{I}) 
    - \frac{1-2\tau\alpha}{1+2\tau\alpha} \bm{I}
   \right)
   \\
 &= \det\!\left(
    \mu^{2} \bm{I} -\mu \bm{Q} - \frac{1-2\tau\alpha}{1+2\tau\alpha} \bm{I} 
   \right).
\end{aligned}
\end{equation}

To proceed from here, it is reasonable to involve the eigenvalues of
the matrix~$\bm{Q}$. As $\bm{Q}$ is real-valued and symmetric, there exist an 
orthogonal matrix $\bm{V}$ of eigenvectors of $\bm{Q}$ and a diagonal matrix 
$\bm{\Gamma}$ with the eigenvalues $\gamma$ of $\bm{Q}$ on its diagonal such 
that
\begin{equation}
 \bm{Q} = \bm{V} \bm{\Gamma} \bm{V}^{T}.
\end{equation}
As $\bm{V}$ is an orthogonal matrix, it holds that
\begin{equation}
 \bm{I} = \bm{V} \bm{V}^{T}.
\end{equation}
Plugging both of these relations into \eqref{eq:ev_B_01} yields
\begin{equation}
 \label{eq:ev_B_02}
 \resizebox{\linewidth}{!}{$
 \begin{aligned}
 \det\!\left(\bm{B} - \mu \bm{I}\right)
 &=
   \det\!\left(
     \mu^{2} \bm{V} \bm{V}^{T}
    -\mu \bm{V} \bm{\Gamma} \bm{V}^{T} 
    - \frac{1-2\tau\alpha}{1+2\tau\alpha} \bm{V} \bm{V}^{T} 
   \right)
 \\
 &= \det\!\left(
    \bm{V}
     \left( 
       \mu^{2} \bm{I}
      -\mu  \bm{\Gamma}  
      - \frac{1-2\tau\alpha}{1+2\tau\alpha} \bm{I} 
     \right)
    \bm{V}^{T} 
   \right)
 \\  
 &=\det\!\left(\bm{V}\right) 
   \det\!\left(
    \mu^{2} \bm{I} -\mu  \bm{\Gamma} - \frac{1-2\tau\alpha}{1+2\tau\alpha} 
    \bm{I} 
   \right)
   \det\!\left(\bm{V}^{T}\right)
 \\
 &= \det\!\left(
    \mu^{2} \bm{I} -\mu  \bm{\Gamma} - \frac{1-2\tau\alpha}{1+2\tau\alpha} 
    \bm{I} 
   \right),
 \end{aligned}
 $}
\end{equation}
since the determinants of the orthogonal matrices $\bm{V}$ and $\bm{V}^{T}$ are
both 1. The remaining determinant is concerned with a 
diagonal matrix. Thus, it is equal to the product of the diagonal elements, i.e.
\begin{equation}
 \det\!\left(\bm{B} - \mu \bm{I}\right)
 = \prod_{j=1}^{N}
   \left( 
    \mu^{2} -\mu\, \gamma_{j} - \frac{1-2\tau\alpha}{1+2\tau\alpha}
   \right).
\end{equation}
For $\mu$ to be an eigenvalue of $\bm{B}$, we need that this product 
vanishes. This is exactly the case if one or more factors in the 
product vanish. Hence, we get that for a fixed eigenvalue $\gamma$ of $\bm Q$, 
an eigenvalue $\mu$ of $\bm B$ satisfies 
\begin{equation}
 \mu = \frac{\gamma}{2} 
       \pm \sqrt{\frac{\gamma^{2}}{4} + \frac{1-2\tau\alpha}{1+2\tau\alpha}}.
\end{equation}
For stability of \eqref{eq:DF_multi}, we need that $|\mu | \leq 1$ for all 
solutions $\mu$ for all eigenvalues $\gamma$ of $\bm{Q}$. To proceed further, 
we 
consider the discriminant 
$\frac{\gamma^{2}}{4} + \frac{1-2\tau\alpha}{1+2\tau\alpha}$. Since we know that
$\bm{Q}$ is real-valued and symmetric, it follows that $\gamma$ is a real
number. Thus, $\gamma^{2}$ is positive. We also know that $\tau$ and $\alpha$ 
are positive, such that
\begin{equation}
 -1 < \frac{1-2\tau\alpha}{1+2\tau\alpha} < 1.
\end{equation}
Therefore, it is possible for the discriminant to have negative values, which 
results in complex eigenvalues $\mu$. Let us therefore distinguish 
the three cases of negative discriminant, vanishing 
discriminant, and positive discriminant. We will see that the last one is the 
only case which introduces a lower bound on $\alpha$ for unconditional 
stability. 

\textbf{Vanishing Discriminant.}
First, we consider a vanishing discriminant, which can only happen if 
$2\tau\alpha \geq 1$. This yields
\begin{equation}
\label{eq:ev_solution}
\gamma = \pm \, 2 \,\sqrt{\frac{2\tau\alpha-1}{2\tau\alpha+1}}.
\end{equation}
Thus, the eigenvalues of $\bm{B}$ are given by 
\begin{equation}
 \mu = \frac{\gamma}{2} = \pm \sqrt{\frac{2\tau\alpha-1}{2\tau\alpha+1}}.
\end{equation}
Since the fraction takes values between $0$ and $1$, the same is true for the
square root. Therefore, we have $|\mu| < 1$.

\textbf{Negative Discriminant.}
If the discriminant is negative, the 
corresponding values of $\mu$ are complex and we can write
\begin{equation}
 \mu = \frac{\gamma}{2} 
       \pm \mathrm{i} \,
       \sqrt{-\frac{\gamma^{2}}{4} - \frac{1-2\tau\alpha}{1+2\tau\alpha}}.
\end{equation}
Then we get for the squared absolute value of $\mu$
\begin{equation}
 \label{eq:complex_mu}
 \left\lvert \mu \right\rvert^{2} =
 \frac{\gamma^2}{4}
 +
 \left(-\frac{\gamma^{2}}{4} - \frac{1-2\tau\alpha}{1+2\tau\alpha}\right)
 =
 \frac{2\tau\alpha-1}{2\tau\alpha+1}.
\end{equation}
Surprisingly, this does not depend on $\gamma$ and we recover once more the
condition
\begin{equation}
 -1 < \frac{2\tau\alpha-1}{2\tau\alpha+1} < 1.
\end{equation}
Thus, we again have $|\mu| < 1$.

\textbf{Positive Discriminant.}
In this case, the eigenvalues $\mu$ are real-valued. Thus, we get the condition
\begin{equation}
 \label{eq:cond_xi01}
 -1 < 
 \frac{\gamma}{2} 
 \pm \sqrt{\frac{\gamma^{2}}{4} + \frac{1-2\tau\alpha}{1+2\tau\alpha}}
 < 1.
\end{equation}
If $\tau > 0$ and $\alpha > 0$, this system has the following solutions for 
$\gamma$: 
\begin{equation}
 \label{eq:cond_xi}
 \begin{aligned}
 &|\gamma| < \frac{4\tau\alpha}{1+2\tau\alpha}, \quad 
   &\text{if } \;\quad  0 \leq \frac{1-2\tau\alpha}{1+2\tau\alpha} < 
   1,
 \\
 \sqrt{\frac{8\tau\alpha-4}{2\tau\alpha+1}} \leq \, &|\gamma|
    < \frac{4\tau\alpha}{1+2\tau\alpha}, \quad 
  &\text{if }  -1 < \frac{1-2\tau\alpha}{1+2\tau\alpha} < 0.
 \end{aligned}
\end{equation}
If we treat $\gamma$ as a complex number for the moment, the first condition 
means that $\gamma$ has to be inside a disc of radius 
$\frac{4\tau\alpha}{1+2\tau\alpha}$ around the origin and the second condition
means that $\gamma$ has to be inside that disc, but outside or on the boundary
of a second disk of radius 
$\sqrt{\frac{8\tau\alpha-4}{2\tau\alpha+1}}$. Hence, the second condition is 
more restrictive.

It remains to determine conditions on $\tau$ and $\alpha$ such that 
\eqref{eq:cond_xi} is always satisfied for all eigenvalues
$\gamma$ of the matrix $\bm{Q}$. As the matrix $\bm{A}(\bm u^k)$ is real-valued 
and symmetric, we can diagonalise it in the same fashion as $\bm{Q}$: There 
exist an orthogonal matrix $\bm W$ and a diagonal matrix $\bm \Lambda$ with the 
eigenvalues $\lambda$ on its diagonal such that 
\begin{equation}
  \bm A(\bm u^k) = \bm W \bm \Lambda \bm W^\top.
\end{equation}
With the help of this representation, we can write 
\begin{equation}
\begin{aligned}
 \bm{Q} 
 &= \frac{4 \tau \alpha}{1 + 2 \tau \alpha} \bm{I}
   - \frac{2 \tau }{1 + 2 \tau \alpha} 
   \bm{A}\!\left(\bm{u}^{k}\right)
 \\
 &= \frac{4 \tau \alpha}{1 + 2 \tau \alpha} \bm{W} \bm{W}^{T}
 - \frac{2 \tau }{1 + 2 \tau \alpha} 
   \bm{W} \bm{\Lambda} \bm{W}^{T}
 \\
 &=\bm{W} \left(\frac{4 \tau \alpha}{1 + 2 \tau \alpha} \bm{I}
                - \frac{2 \tau }{1 + 2 \tau \alpha} \bm{\Lambda} \right)
   \bm{W}^{T}.
\end{aligned}
\end{equation}
Hence, the eigenvalues of $\bm{Q}$ are given by 
\begin{equation}
 \label{eq:ev_M_from_sys}
 \gamma = \frac{4 \tau \alpha}{1 + 2 \tau \alpha}
  - \frac{2 \tau }{1 + 2 \tau \alpha} \lambda,
\end{equation}
where $\lambda$ is an eigenvalue of $\bm{A}(\bm{u}^{k})$.

With this formula, the first case of \eqref{eq:cond_xi} reduces to
\begin{equation}
 \label{eq:cond_lambda01}
 0 < \lambda < 4 \alpha.
\end{equation}
This condition is similar to the stability condition of the explicit scheme, 
however now for $\alpha$ instead of $\frac{1}{\tau}$. The estimate on the 
left-hand side is always fulfilled if $\bm A(\bm u^k)$ is positive 
definite. For now, we assume that $\bm A(\bm u^k)$ is positive definite and 
consider the case of a zero eigenvalue afterwards.

The inequality \eqref{eq:cond_lambda01} has to hold for every eigenvalue of
$\bm A(\bm u^k)$. If $\bm A(\bm u^k)$ is positive definite, we can replace 
$\lambda$ by the spectral radius $\rho(\bm A(\bm u^k))$ as an upper bound. 
Hence, we arrive at 
\begin{equation}
 \label{eq:stab_cond}
 \alpha > \frac{\rho\!\left(\bm A(\bm u^k)\right)}{4}.
\end{equation}
This is in line with the result that has been derived for the linear case 
\cite{GG76}. 

For the second case in \eqref{eq:cond_xi}, plugging in 
\eqref{eq:ev_M_from_sys} and simplifying yields
\begin{equation}
 \label{eq:cond_lambda02}
 \begin{aligned}
  &0 < \lambda \leq 
  2 \alpha - \sqrt{4\alpha^{2}-\frac{1}{\tau^{2}}}
  \qquad \text{or} \\ 
  &2 \alpha + \sqrt{4\alpha^{2}-\frac{1}{\tau^{2}}}
  \leq 
  \lambda
  <
  4\alpha 
 \end{aligned}
\end{equation}
As we are in the case in which $-1 < \frac{1-2\tau\alpha}{1+2\tau\alpha} < 
0$, the square root is a non-negative real number. 

Moreover, we have to investigate what happens if
\begin{equation}\label{eq:lam_range}
 \begin{aligned}
  &-1 \leq \frac{1-2\tau\alpha}{1+2\tau\alpha} < 0
  \qquad \text{and} \\ 
 &2 \alpha - \frac{\sqrt{4\alpha^{2}\tau^{2}-1}}{\tau^{2}}
 < \lambda < 
 2 \alpha + \frac{\sqrt{4\alpha^{2}\tau^{2}-1}}{\tau^{2}}.
 \end{aligned}
\end{equation}
With some tedious but straight-forward computations, one can show that this 
case corresponds exactly to the case with a negative discriminant, which did 
not introduce any new stability conditions. 

Lastly, we consider the case where $\bm A(\bm u^k)$ is only positive 
semidefinite. If $2\tau \alpha < 1$, then $\lambda=0$ lies in the 
range given for $\lambda$ in \eqref{eq:lam_range}. For this case, we have 
already shown that $|\mu|<1$ and the scheme is stable. 

If $2\tau\alpha>1$, we obtain from \eqref{eq:ev_M_from_sys}:
\begin{equation}
  \gamma = \frac{4 \tau \alpha}{1 + 2 \tau \alpha}.
\end{equation}
We can plug this value of $\gamma$ into \eqref{eq:ev_solution} to see that 
$\mu=1$ 
is always a solution in this case. We have to ensure that all other solutions 
for $\mu$ fulfil $|\mu|<1$. The other solution of \eqref{eq:ev_solution} 
for $\lambda=0$ is
\begin{equation}
  \mu = \frac{2 \tau \alpha - 1}{2 \tau \alpha + 1},
\end{equation}
which yields $|\mu| < 1$ since $\tau > 0$ and $\alpha > 0$. All other 
eigenvalues of $\bm B$ have an absolute value of less than one by the 
considerations above. Thus, we can conclude that also in this case, the Du 
Fort--Frankel scheme is stable. 

Consequently, the only stability condition on $\alpha$ is the one in 
\eqref{eq:stab_cond}. In a similar manner as for the symmetric ResNet, we can 
transform this condition into an a priori constraint
\begin{equation}
  \alpha \geq \frac{L}{4} .
\end{equation}


%
%

\bibliographystyle{spmpsci}       
\bibliography{myrefs}         

\end{document}

%% file: flowchart_framework.tex
\tikzset
{
  block/.style = {shape=rectangle, rounded corners,
                  draw, anchor=center, minimum height = 0em,
                  align=center, minimum width = 2em,
                  inner sep=1ex},
  operator/.style = {block, circle, inner sep = .3ex, minimum width = 0em, 
                     minimum height = 0em},
  summary/.style = {block, fill=black!10, rounded corners=0cm, minimum height = 
  0em,}
}

\newcommand{\edge}[2]{\draw (#1) edge node (TMP) {} (#2);}

\newcommand{\layername}[4]{\path (#1) -- node (TMP) {} (#2);
                           \draw node [left of=TMP, xshift=#4] {#3};}

\newcommand{\edgelabel}[1]{\draw node [right of=TMP, anchor=west, 
                           xshift=-2.2em] {#1};}
                         
\newcommand{\edgelabelleft}[1]{\draw node [left of=TMP, anchor=east, 
                               xshift=+2.2em] {#1};}

\newcommand{\skipedge}[3] {\draw (#1) -- +(#3,0) |-
                           node (TMP) {}
                           node [near start,left] {Id} (#2);}
                             
\newcommand{\emptyskipedge}[3] {\draw (#1) -- +(#3,0) |-
                           node (TMP) {}
                           (#2);}

%% file: diffusion_block.tex
\begin{tikzpicture}[-latex]
  \matrix (m)
  [
    matrix of nodes,
    column 1/.style = {nodes={block}}
  ]
  {
    |(input)|          
    $\bm u^k$
    \\[4.5ex]            
    |(fwd)|
    $\bm K \bm u^k$%
    \\[4.5ex]
    |(flux)|
    $\tau \, \bm\Phi\!\left(\bm K \bm u^k\right)$%
    \\[4.5ex]
    |(bwd)|
    $-\tau \, \bm K^\top \bm\Phi\!\left(\bm K \bm u^k\right)$%
    \\[2ex]
    |[operator] (add)|
    $+$%
    \\[2ex]
    |(output)|         
    $\bm u^{k+1}$%
    \\
  };

    \def\nameshift{-1.8em}                    
    \def\skipconnshift{2.0}                     

    \layername{input}{fwd}{convolution}{\nameshift}   
    \edge{input}{fwd}
    \edgelabel{$\bm K$}
    
    \layername{fwd}{flux}{activation}{\nameshift}      
    \edge{fwd}{flux}     
    \edgelabel{$\tau \bm\Phi(\cdot)$}           
                                              
    \layername{flux}{bwd}{convolution}{\nameshift}                             
    \edge{flux}{bwd}          
    \edgelabel{$-\bm K^\top$}

    \node[left of=add, xshift=-4ex] {skip connection};
    \edge{bwd}{add}                             

    \edge{add}{output}
  
    \skipedge{input}{add}{\skipconnshift}   

    \node[above of=input, yshift=-1ex] (TMP) {};
    \draw (TMP) -- (input);

    \node[below of=output, yshift=1ex] (TMP) {};
    \draw (output) -- (TMP);

    \node[left of=input, xshift=\nameshift-3.8em, yshift=2em] (topleft) {};
    \node[right of=output, xshift=4em, yshift=-1.6em] (bottomright) {};
    \draw (topleft) rectangle (bottomright);

    \node [left of=flux, rotate=90, anchor=north, yshift=7em]
          {\textbf{Diffusion Block}};
\end{tikzpicture}

%% file: activations.tex
\begin{center}
\setlength{\tabcolsep}{3mm}
\resizebox{\linewidth}{!}{
\begin{tabular}{ccc}     
    \begin{tikzpicture}[scale=0.6]
      \def\la{0.6}
      \begin{axis}
      [
        font=\Large,
        samples=300,
        domain=-10:10, 
        axis lines=middle,
        xmin = -2,
        xmax = 2,
        xtick = {1},
        ytick = {1},
        inner sep=0pt,
        tick style = {thick, black},
        major tick length = 7,
        x label style={at={(axis description cs:1.0,0.44)},anchor=north},
        y label style={at={(axis description cs:0.4,1.0)},anchor=north},
        x tick label style={yshift=-1mm},
        y tick label style={xshift=-1mm},
        width=\axisdefaultwidth,
        height=0.8*\axisdefaultwidth,
        xlabel = {$s$},
        ylabel = {$\sigma(s)$},
        ymin = -2,  
        ymax = 2, 
        xtick = {1},
        ytick = {1},
        xticklabels = {$1$},
        yticklabels = {$1$}
      ]
      
      \addplot[blue, thick]
        plot [domain=-10:0] (\x, 0);
      \addplot[blue, thick]
        plot [domain=0:10] (\x, \x);
      
      \end{axis}
    \end{tikzpicture}
  &
      \begin{tikzpicture}[scale=0.6]
        \def\la{0.6}
        \pgfmathsetmacro{\ymarki}{\la / sqrt(2)}
        \begin{axis}
        [
          font=\Large,
          samples=300,
          domain=-10:10, 
          axis lines=middle,
          xmin = -2,
          xmax = 2,
          xtick = {1},
          ytick = {1},
          inner sep=0pt,
          tick style = {thick, black},
          major tick length = 7,
          x label style={at={(axis description cs:1.0,0.44)},anchor=north},
          y label style={at={(axis description cs:0.4,1.0)},anchor=north},
          x tick label style={yshift=-1mm},
          y tick label style={xshift=-1mm},
          width=\axisdefaultwidth,
          height=0.8*\axisdefaultwidth,
          xlabel = {$s$},
          ylabel = {$\sigma(s)$},
          ymin = -1,  
          ymax = 1, 
          xtick = {\la},
          ytick = {\ymarki},
          xticklabels = {$\lambda$},
          yticklabels = {$\frac{\lambda}{\sqrt{2}}$}
        ]
        
        \addplot[blue, thick]
          plot (\x, {\x / sqrt(1 + \x^2 / \la^2)});
        
        \end{axis}
      \end{tikzpicture}
  &
    \begin{tikzpicture}[scale=0.6]
      \def\la{0.6}
          
      \pgfmathsetmacro{\ymarki}{\la / 2}
      \begin{axis}
       [
        font=\Large,
        samples=300,
        domain=-10:10, 
        axis lines=middle,
        xmax = 4, 
        xmin = -4, 
        ymax = 0.8, 
        ymin = -0.8, 
        xtick = {1},
        ytick = {1},
        inner sep=0pt,
        tick style = {thick, black},
        major tick length = 7,
        x label style={at={(axis description cs:1.0,0.44)},anchor=north},
        y label style={at={(axis description cs:0.4,1.0)},anchor=north},
        x tick label style={yshift=-1mm},
        y tick label style={xshift=-1mm},
        width=\axisdefaultwidth,
        height=0.8*\axisdefaultwidth,
        xtick = {\la},
        ytick = {\ymarki},
        xticklabels = {$\lambda$},
        yticklabels = {$\frac{\lambda}{2}$},
        xlabel = {$s$},
        ylabel = {$\sigma(s)$},
      ]
      
      \addplot[blue, thick]
          plot (\x, {\x / (1 + \x^2 / \la^2)});
      
      \end{axis}
    \end{tikzpicture}
  
  \\[2ex]
    ReLU activation
  & Charbonnier activation
  & Perona--Malik activation
  \\

  & \makecell{(from convex energy)} 
  & \makecell{(from nonconvex energy)}
  \\[1ex]
  $\sigma(s) = \text{max}(0,s)$
  &$\sigma(s) = \frac{s}{\sqrt{1+\frac{s^2}{\lambda^2}}}$
  &$\sigma(s) = \frac{s}{1+\frac{s^2}{\lambda^2}}$
  \\[2ex]
  \multicolumn{2}{c}{$\underbrace{\phantom{\hspace{8cm}}}$}
  & 
  \multicolumn{1}{c}{$\underbrace{\phantom{\hspace{4cm}}}$}
  \\[2ex]
  \multicolumn{2}{c}{well-known, widely used in CNNs}
  &
  \multicolumn{1}{c}{\makecell{unexpected, \\ hardly used in CNNs}}
\end{tabular}
}
\end{center}

%% file: fsi_block.tex
\begin{tikzpicture}[-latex]
  \matrix (m)
  [
    matrix of nodes,
    column 1/.style = {nodes={block}}
  ]
  {
    |(input)|          
    $\bm u^{k + \frac{\ell - 1}{L}}$ 
    \\[2ex]            
    |[summary] (diff1)|
    diffusion block%
    \\[4ex]
    |[operator] (add1)|
    $+$%
    \\[2ex]
    |(middle)|
    $\bm u^{k + \frac{\ell}{L}}$
    \\[2ex]
    |[summary] (diff2)|
    diffusion block%
    \\[4ex]
    |[operator] (add2)|
    $+$%
    \\[2ex]
    |(output)|
    $\bm u^{k + \frac{\ell + 1}{L}}$
    \\
  };

    \def\nameshift{-1em}                    
    \def\skipconnshift{2.0}                     
  
    \edge{input}{diff1}
    \edge{diff1}{add1}
    \edge{add1}{middle}
    \edge{middle}{diff2}
    \edge{diff2}{add2}
    \edge{add2}{output}
    
    \emptyskipedge{input}{add2}{\skipconnshift}
    \node[right of=add2, yshift=1.5ex] {$\cdot (1-\alpha_{\ell})$};
    \node[above of=add2, xshift=-2.5ex, yshift=-3.5ex] {$\cdot \alpha_{\ell}$};
    
    \node[right of=output, xshift=8.5ex, yshift=-6.2ex] (TMP) {} ;
    \draw[thin, dashed, black!65] (middle) -- +(2.1,0) -- (TMP);
    \node[thin, dashed, right of=add1, yshift=1.5ex, xshift=1ex] 
    {$\cdot (1-\alpha_{\ell-1})$};
    \node[thin, dashed, above of=add1, xshift=-3ex, yshift=-3.5ex] 
    {$\cdot \alpha_{\ell-1}$};
        
    \node[right of=input, xshift=8.5ex, yshift=6.2ex] (TMP) {} ;
    \draw[thin, dashed, black!65] (TMP) |- (add1);

    \node[above of=input, yshift=-1ex] (TMP) {};
    \draw (TMP) -- (input);

    \node[below of=output, yshift=1ex] (TMP) {};
    \draw (output) -- (TMP);

    \node[left of=input, xshift=-3em, yshift=2em] (topleft) {};
    \node[right of=output, xshift=4em, yshift=-1.6em] (bottomright) {};
    \draw (topleft) rectangle (bottomright);

    \node [left of=middle, rotate=90, anchor=north, yshift=4.5em]
          {\textbf{FSI Block}};
\end{tikzpicture}

%% file: unet.tex
\begin{tikzpicture}[-latex]
  \matrix (m)
  [
    matrix of nodes,
    column sep=3mm,
    column 1/.style = {nodes={block}},
    column 2/.style = {nodes={block}},
    column 3/.style = {nodes={block}},
    column 4/.style = {nodes={block}},
    column 5/.style = {nodes={block}},
    column 6/.style = {nodes={block}},
    column 7/.style = {nodes={block}},
    column 8/.style = {nodes={block}}
  ]
  {
        |(input)|          
        $\bm f^h$%
    &   
        |[summary] (conv1)|
        $\bm C_1^h(\cdot)$%
    &   
        |(conv1out)|
        $\bm{\tilde f}^h$%
    &   
        %
        %
    &   
        |[operator] (add)|
        $+$%
    &   
        |(conv2in)|
        $\bm{\tilde f}_{\text{new}}^h$%
    &   
        |[summary] (conv2)|
        $\bm C_2^h(\cdot)$%
    &   
        |(output)|
        $\bm{\hat{f}}^h$%
    \\[5ex]
        %
        %
    &   
        %
        %
    &   
        |(down)|
        $\bm f^H$%
    &   
        |[summary] (downconv)|
        $\bm C^H(\cdot)$%
    &   
        |(downconvout)|
        $\bm{\tilde f}^H$%
    &   
        %
        %
    &   
        %
        %
    &   
    \\
  };
  
  \edge{input}{conv1}
  \edge{conv1}{conv1out}
  \edge{conv1out}{add}
  \edge{add}{conv2in}
  \edge{conv2in}{conv2}
  \edge{conv2}{output}
  
  \edge{conv1out}{down}
  \edgelabelleft{$\bm R^{h\rightarrow H}$}
  \edge{down}{downconv}
  \edge{downconv}{downconvout}
  \edge{downconvout}{add}
  \edgelabel{$\bm P^{H \rightarrow h}$}

\end{tikzpicture}

%% file: fas.tex
\begin{tikzpicture}[-latex]
  \matrix (m)
  [
    matrix of nodes,
    column sep=3mm,
    column 1/.style = {nodes={block}},
    column 2/.style = {nodes={block}},
    column 3/.style = {nodes={block}},
    column 4/.style = {nodes={block}},
    column 5/.style = {nodes={block}},
    column 6/.style = {nodes={block}},
    column 7/.style = {nodes={block}},
    column 8/.style = {nodes={block}},
  ]
  {
        |(input)|          
        $\bm x_0^h, \bm y^h, \bm b^h, \cdot$%
    &   
        |[summary] (solv1)|
        $\bm S^h_{\bm A}(\cdot)$%
    &   
        |(solv1out)|
        $\bm{\tilde x}^h, \bm y^h, \bm b^h, \bm r^h$%
    &   
        %
        %
    &   
        |[operator] (add)|
        $+$%
    &   
        |(solv2in)|
        $\tilde{\bm{x}}_{\text{new}}^h, \bm y^h, \bm b^h, \cdot$%
    &   
        |[summary] (solv2)|
        $\bm S^h_{\bm A}(\cdot)$%
    &   
        |(output)|
        $\bm{\hat x}^h, \bm y^h, \bm b^h, \bm{\hat{r}}^h$%
    \\[15ex]
        %
        %
    &   
        %
        %
    &   
        |(down)|
        $\bm x_0^H, \bm y^H, \bm b^H, \cdot$%
    &   
        |[summary] (downsolv)|
        $\bm S^H_{\bm A}(\cdot)$%
    &   
        |(downsolvout)|
        $\bm{\tilde x}^H, \bm y^H, \bm b^H, \bm r^H$%
    &   
        %
        %
    &   
        %
        %
    &   
    \\
  };
  
  \edge{input}{solv1}
  \edge{solv1}{solv1out}
  \edge{solv1out}{add}
  \edge{add}{solv2in}
  \edge{solv2in}{solv2}
  \edge{solv2}{output}
  
  \edge{solv1out}{down}
  \edgelabelleft{$\begin{Bmatrix}
  \bm 0 & \bm 0 & \bm 0 & \bm 0 \\
  \bm R^{h \rightarrow H} & \bm 0 & \bm 0 & \bm 0  \\
  \bm 0 & \bm 0 & \bm 0 & \bm R^{h \rightarrow H} \\
  \cdot & \cdot & \cdot & \cdot
  \end{Bmatrix}$}
  \edge{down}{downsolv}
  \edge{downsolv}{downsolvout}
  \edge{downsolvout}{add}
  \edgelabel{$\begin{Bmatrix}
  \bm P^{H \rightarrow h} & -\bm P^{H \rightarrow h} &\bm 0 & \bm 0\\
  \bm 0 & \bm 0 & \bm 0 & \bm 0\\
  \bm 0 & \bm 0 & \bm 0 & \bm 0\\
  \cdot & \cdot & \cdot & \cdot
  \end{Bmatrix}$}

\end{tikzpicture}

%% file: exp_atomic_relu.tex
\begin{tikzpicture}[scale = 1.0]

\begin{axis}
[
width=0.9*\axisdefaultwidth,
height=0.75*\axisdefaultwidth,
samples=500,
axis lines=left,
ymin = 30.5, ymax = 38,
xmin = 0, xmax = 22,
xtick = {1, 5, 10, 15,20},
ytick = {31, 33, ..., 37},
ylabel={PSNR},
xlabel={number of residual blocks},
xlabel near ticks,
ylabel near ticks,
grid = major,
legend entries={ResNet,
                SymResNet, 
                Du Fort--Frankel,
                FSI,
                Linear diffusion},
title = {\textbf{ReLU activation}},
legend cell align=left,
legend style={at={(0.3, 0.7)}, anchor = west}
]

\addplot+[thick, dotted, mark options={solid}, black, mark = *,
restrict expr to domain={\thisrow{modelid}}{\STANDARDRESNET:\STANDARDRESNET},
restrict expr to domain={\thisrow{activationid}}{\RELU:\RELU},
restrict expr to domain={\thisrow{couplingid}}{\UNCOUPLED:\UNCOUPLED},
restrict expr to domain={\thisrow{channels}}{1:1},
restrict expr to domain={\thisrow{smoothness}}{100000000:100000000},
unbounded coords=discard]
table[x = steps, y = psnr]
{test_errors};

\addplot+[thick, dashed, mark options={solid}, red, mark = square*,
restrict expr to domain={\thisrow{modelid}}{\STABLERESNET:\STABLERESNET},
restrict expr to domain={\thisrow{activationid}}{\RELU:\RELU},
restrict expr to domain={\thisrow{couplingid}}{\UNCOUPLED:\UNCOUPLED},
restrict expr to domain={\thisrow{channels}}{1:1},
restrict expr to domain={\thisrow{smoothness}}{100000000:100000000},
unbounded coords=discard]
table[x = steps, y = psnr]
{test_errors};

\addplot+[thick, dashdotted, mark options={solid}, blue, mark = diamond*,
restrict expr to domain={\thisrow{modelid}}{\DUFORTFRANKEL:\DUFORTFRANKEL},
restrict expr to domain={\thisrow{activationid}}{\RELU:\RELU},
restrict expr to domain={\thisrow{couplingid}}{\UNCOUPLED:\UNCOUPLED},
restrict expr to domain={\thisrow{channels}}{1:1},
restrict expr to domain={\thisrow{smoothness}}{100000000:100000000},
unbounded coords=discard]
table[x = steps, y = psnr]
{test_errors};

\addplot+[thick, densely dotted, mark options={solid}, black!60!green, mark = 
star,
restrict expr to domain={\thisrow{modelid}}{\FSI:\FSI},
restrict expr to domain={\thisrow{activationid}}{\RELU:\RELU},
restrict expr to domain={\thisrow{couplingid}}{\UNCOUPLED:\UNCOUPLED},
restrict expr to domain={\thisrow{channels}}{1:1},
restrict expr to domain={\thisrow{smoothness}}{100000000:100000000},
unbounded coords=discard]
table[x = steps, y = psnr]
{test_errors};

\addplot+[dashed, black, thick, mark=none, domain=0:100] plot (\x,\bestHDpsnr);
\end{axis}

\end{tikzpicture}

%% file: exp_atomic_charb.tex
\begin{tikzpicture}[scale = 1.0]

\begin{axis}
[
width=0.9*\axisdefaultwidth,
height=0.75*\axisdefaultwidth,
samples=500,
axis lines=left,
ymin = 30.5, ymax = 38,
xmin = 0, xmax = 22,
xtick = {1, 5, 10, 15,20},
ytick = {31, 33, ..., 37},
ylabel={PSNR},
xlabel={number of residual blocks},
xlabel near ticks,
ylabel near ticks,
grid = major,
legend entries={ResNet,
                SymResNet, 
                Du Fort--Frankel,
                FSI,
                Charbonnier diffusion},
title = {\textbf{Charbonnier activation}},
legend cell align=left,
legend style={at={(0.3, 0.25)}, anchor = west}
]

\addplot+[thick, dotted, mark options={solid}, black, mark = *,
restrict expr to domain={\thisrow{modelid}}{\STANDARDRESNET:\STANDARDRESNET},
restrict expr to domain={\thisrow{activationid}}{\CHARBONNIER:\CHARBONNIER},
restrict expr to domain={\thisrow{couplingid}}{\UNCOUPLED:\UNCOUPLED},
restrict expr to domain={\thisrow{channels}}{1:1},
restrict expr to domain={\thisrow{smoothness}}{100000000:100000000},
unbounded coords=discard]
table[x = steps, y = psnr]
{test_errors};

\addplot+[thick, dashed, mark options={solid}, red, mark = square*,
restrict expr to domain={\thisrow{modelid}}{\STABLERESNET:\STABLERESNET},
restrict expr to domain={\thisrow{activationid}}{\CHARBONNIER:\CHARBONNIER},
restrict expr to domain={\thisrow{couplingid}}{\UNCOUPLED:\UNCOUPLED},
restrict expr to domain={\thisrow{channels}}{1:1},
restrict expr to domain={\thisrow{smoothness}}{100000000:100000000},
unbounded coords=discard]
table[x = steps, y = psnr]
{test_errors};

\addplot+[thick, dashdotted, mark options={solid}, blue, mark = diamond*,
restrict expr to domain={\thisrow{modelid}}{\DUFORTFRANKEL:\DUFORTFRANKEL},
restrict expr to domain={\thisrow{activationid}}{\CHARBONNIER:\CHARBONNIER},
restrict expr to domain={\thisrow{couplingid}}{\UNCOUPLED:\UNCOUPLED},
restrict expr to domain={\thisrow{channels}}{1:1},
restrict expr to domain={\thisrow{smoothness}}{100000000:100000000},
unbounded coords=discard]
table[x = steps, y = psnr]
{test_errors};

\addplot+[thick, densely dotted, mark options={solid}, black!60!green, mark = 
star,
restrict expr to domain={\thisrow{modelid}}{\FSI:\FSI},
restrict expr to domain={\thisrow{activationid}}{\CHARBONNIER:\CHARBONNIER},
restrict expr to domain={\thisrow{couplingid}}{\UNCOUPLED:\UNCOUPLED},
restrict expr to domain={\thisrow{channels}}{1:1},
restrict expr to domain={\thisrow{smoothness}}{100000000:100000000},
unbounded coords=discard]
table[x = steps, y = psnr]
{test_errors};

\addplot+[dashed, black, thick, mark=none, domain=0:100] plot (\x,\bestCHpsnr);
\end{axis}

\end{tikzpicture}

%% file: exp_atomic_pm.tex
\begin{tikzpicture}[scale = 1.0]

\begin{axis}
[
width=0.9*\axisdefaultwidth,
height=0.75*\axisdefaultwidth,
samples=500,
axis lines=left,
ymin = 30.5, ymax = 38,
xmin = 0, xmax = 22,
xtick = {1, 5, 10, 15,20},
ytick = {31, 33, ..., 37},
ylabel={PSNR},
xlabel={number of residual blocks},
xlabel near ticks,
ylabel near ticks,
grid = major,
legend entries={ResNet,
                SymResNet, 
                Du Fort--Frankel,
                FSI,
                Perona--Malik diffusion},
title = {\textbf{Perona--Malik activation}},
legend cell align=left,
legend style={at={(0.3, 0.25)}, anchor = west}
]

\addplot+[thick, dotted, mark options={solid}, black, mark = *,
restrict expr to domain={\thisrow{modelid}}{\STANDARDRESNET:\STANDARDRESNET},
restrict expr to domain={\thisrow{activationid}}{\PERONAMALIK:\PERONAMALIK},
restrict expr to domain={\thisrow{couplingid}}{\UNCOUPLED:\UNCOUPLED},
restrict expr to domain={\thisrow{channels}}{1:1},
restrict expr to domain={\thisrow{smoothness}}{100000000:100000000},
unbounded coords=discard]
table[x = steps, y = psnr]
{test_errors};

\addplot+[thick, dashed, mark options={solid}, red, mark = square*,
restrict expr to domain={\thisrow{modelid}}{\STABLERESNET:\STABLERESNET},
restrict expr to domain={\thisrow{activationid}}{\PERONAMALIK:\PERONAMALIK},
restrict expr to domain={\thisrow{couplingid}}{\UNCOUPLED:\UNCOUPLED},
restrict expr to domain={\thisrow{channels}}{1:1},
restrict expr to domain={\thisrow{smoothness}}{100000000:100000000},
unbounded coords=discard]
table[x = steps, y = psnr]
{test_errors};

\addplot+[thick, dashdotted, mark options={solid}, blue, mark = diamond*,
restrict expr to domain={\thisrow{modelid}}{\DUFORTFRANKEL:\DUFORTFRANKEL},
restrict expr to domain={\thisrow{activationid}}{\PERONAMALIK:\PERONAMALIK},
restrict expr to domain={\thisrow{couplingid}}{\UNCOUPLED:\UNCOUPLED},
restrict expr to domain={\thisrow{channels}}{1:1},
restrict expr to domain={\thisrow{smoothness}}{100000000:100000000},
unbounded coords=discard]
table[x = steps, y = psnr]
{test_errors};

\addplot+[thick, densely dotted, mark options={solid}, black!60!green, mark = 
star,
restrict expr to domain={\thisrow{modelid}}{\FSI:\FSI},
restrict expr to domain={\thisrow{activationid}}{\PERONAMALIK:\PERONAMALIK},
restrict expr to domain={\thisrow{couplingid}}{\UNCOUPLED:\UNCOUPLED},
restrict expr to domain={\thisrow{channels}}{1:1},
restrict expr to domain={\thisrow{smoothness}}{100000000:100000000},
unbounded coords=discard]
table[x = steps, y = psnr]
{test_errors};

\addplot+[dashed, black, thick, mark=none, domain=0:100] plot (\x,\bestPMpsnr);4
\end{axis}

\end{tikzpicture}

%% file: exp_dynamic_relu.tex
\begin{tikzpicture}[scale = 1.0]

\begin{axis}
[
width=0.9*\axisdefaultwidth,
height=0.75*\axisdefaultwidth,
samples=500,
axis lines=left,
ymin = 30.5, ymax = 38,
xmin=3.8, xmax = 250, xmode=log,
ytick = {31, 33, ..., 37},
ylabel={PSNR},
xlabel={number of parameters, logarithmic scale},
xlabel near ticks,
ylabel near ticks,
grid = major,
legend entries={ResNet,
                SymResNet, 
                Du Fort--Frankel,
                FSI,
                Linear diffusion},
title = {\textbf{ReLU activation}},
legend cell align=left,
legend style={at={(0.02, 0.8)}, anchor = west}
]

\addplot+[thick, dotted, mark options={solid}, black, mark = *,
restrict expr to domain={\thisrow{modelid}}{\STANDARDRESNET:\STANDARDRESNET},
restrict expr to domain={\thisrow{activationid}}{\RELU:\RELU},
restrict expr to domain={\thisrow{couplingid}}{\UNCOUPLED:\UNCOUPLED},
restrict expr to domain={\thisrow{channels}}{1:1},
restrict expr to domain={\thisrow{smoothness}}{5.0:5.0},
unbounded coords=discard]
table[x = parameters, y = psnr]
{test_errors};

\addplot+[thick, dashed, mark options={solid}, red, mark = square*,
restrict expr to domain={\thisrow{modelid}}{\STABLERESNET:\STABLERESNET},
restrict expr to domain={\thisrow{activationid}}{\RELU:\RELU},
restrict expr to domain={\thisrow{couplingid}}{\UNCOUPLED:\UNCOUPLED},
restrict expr to domain={\thisrow{channels}}{1:1},
restrict expr to domain={\thisrow{smoothness}}{10.0:10.0},
unbounded coords=discard]
table[x = parameters, y = psnr]
{test_errors};

\addplot+[thick, dashdotted, mark options={solid}, blue, mark = diamond*,
restrict expr to domain={\thisrow{modelid}}{\DUFORTFRANKEL:\DUFORTFRANKEL},
restrict expr to domain={\thisrow{activationid}}{\RELU:\RELU},
restrict expr to domain={\thisrow{couplingid}}{\UNCOUPLED:\UNCOUPLED},
restrict expr to domain={\thisrow{channels}}{1:1},
restrict expr to domain={\thisrow{smoothness}}{10.0:10.0},
unbounded coords=discard]
table[x = parameters, y = psnr]
{test_errors};

\addplot+[thick, densely dotted, mark options={solid}, black!60!green, mark = 
star,
restrict expr to domain={\thisrow{modelid}}{\FSI:\FSI},
restrict expr to domain={\thisrow{activationid}}{\RELU:\RELU},
restrict expr to domain={\thisrow{couplingid}}{\UNCOUPLED:\UNCOUPLED},
restrict expr to domain={\thisrow{channels}}{1:1},
restrict expr to domain={\thisrow{smoothness}}{10.0:10.0},
unbounded coords=discard]
table[x = parameters, y = psnr]
{test_errors};

\addplot+[dashed, black, thick, mark=none, domain=0.1:50000] plot 
(\x,\bestHDpsnr);
\end{axis}

\end{tikzpicture}

%% file: exp_dynamic_charb.tex
\begin{tikzpicture}[scale = 1.0]

\begin{axis}
[
width=0.9*\axisdefaultwidth,
height=0.75*\axisdefaultwidth,
samples=500,
axis lines=left,
ymin = 30.5, ymax = 38,
xmin=3.8, xmax = 250, xmode=log,
ytick = {31, 33, ..., 37},
ylabel={PSNR},
xlabel={number of parameters, logarithmic scale},
xlabel near ticks,
ylabel near ticks,
grid = major,
legend entries={ResNet,
                SymResNet, 
                Du Fort--Frankel,
                FSI,
                Charbonnier diffusion},
title = {\textbf{Charbonnier activation}},
legend cell align=left,
legend style={at={(0.3, 0.25)}, anchor = west}
]

\addplot+[thick, dotted, mark options={solid}, black, mark = *,
restrict expr to domain={\thisrow{modelid}}{\STANDARDRESNET:\STANDARDRESNET},
restrict expr to domain={\thisrow{activationid}}{\CHARBONNIER:\CHARBONNIER},
restrict expr to domain={\thisrow{couplingid}}{\UNCOUPLED:\UNCOUPLED},
restrict expr to domain={\thisrow{channels}}{1:1},
restrict expr to domain={\thisrow{smoothness}}{5.0:5.0},
unbounded coords=discard]
table[x = parameters, y = psnr]
{test_errors};

\addplot+[thick, dashed, mark options={solid}, red, mark = square*,
restrict expr to domain={\thisrow{modelid}}{\STABLERESNET:\STABLERESNET},
restrict expr to domain={\thisrow{activationid}}{\CHARBONNIER:\CHARBONNIER},
restrict expr to domain={\thisrow{couplingid}}{\UNCOUPLED:\UNCOUPLED},
restrict expr to domain={\thisrow{channels}}{1:1},
restrict expr to domain={\thisrow{smoothness}}{10.0:10.0},
unbounded coords=discard]
table[x = parameters, y = psnr]
{test_errors};

\addplot+[thick, dashdotted, mark options={solid}, blue, mark = diamond*,
restrict expr to domain={\thisrow{modelid}}{\DUFORTFRANKEL:\DUFORTFRANKEL},
restrict expr to domain={\thisrow{activationid}}{\CHARBONNIER:\CHARBONNIER},
restrict expr to domain={\thisrow{couplingid}}{\UNCOUPLED:\UNCOUPLED},
restrict expr to domain={\thisrow{channels}}{1:1},
restrict expr to domain={\thisrow{smoothness}}{10.0:10.0},
unbounded coords=discard]
table[x = parameters, y = psnr]
{test_errors};

\addplot+[thick, densely dotted, mark options={solid}, black!60!green, mark = 
star,
restrict expr to domain={\thisrow{modelid}}{\FSI:\FSI},
restrict expr to domain={\thisrow{activationid}}{\CHARBONNIER:\CHARBONNIER},
restrict expr to domain={\thisrow{couplingid}}{\UNCOUPLED:\UNCOUPLED},
restrict expr to domain={\thisrow{channels}}{1:1},
restrict expr to domain={\thisrow{smoothness}}{10.0:10.0},
unbounded coords=discard]
table[x = parameters, y = psnr]
{test_errors};

\addplot+[dashed, black, thick, mark=none, domain=0.1:50000] plot 
(\x,\bestCHpsnr);
\end{axis}

\end{tikzpicture}

%% file: exp_dynamic_pm.tex
\begin{tikzpicture}[scale = 1.0]

\begin{axis}
[
width=0.9*\axisdefaultwidth,
height=0.75*\axisdefaultwidth,
samples=500,
axis lines=left,
ymin = 30.5, ymax = 38,
xmin=3.8, xmax = 250, xmode=log,
ytick = {31, 33, ..., 37},
ylabel={PSNR},
xlabel={number of parameters, logarithmic scale},
xlabel near ticks,
ylabel near ticks,
grid = major,
legend entries={ResNet,
                SymResNet, 
                Du Fort--Frankel,
                FSI,
                Perona--Malik diffusion},
title = {\textbf{Perona--Malik activation}},
legend cell align=left,
legend style={at={(0.3, 0.25)}, anchor = west}
]

\addplot+[thick, dotted, mark options={solid}, black, mark = *,
restrict expr to domain={\thisrow{modelid}}{\STANDARDRESNET:\STANDARDRESNET},
restrict expr to domain={\thisrow{activationid}}{\PERONAMALIK:\PERONAMALIK},
restrict expr to domain={\thisrow{couplingid}}{\UNCOUPLED:\UNCOUPLED},
restrict expr to domain={\thisrow{channels}}{1:1},
restrict expr to domain={\thisrow{smoothness}}{5.0:5.0},
unbounded coords=discard]
table[x = parameters, y = psnr]
{test_errors};

\addplot+[thick, dashed, mark options={solid}, red, mark = square*,
restrict expr to domain={\thisrow{modelid}}{\STABLERESNET:\STABLERESNET},
restrict expr to domain={\thisrow{activationid}}{\PERONAMALIK:\PERONAMALIK},
restrict expr to domain={\thisrow{couplingid}}{\UNCOUPLED:\UNCOUPLED},
restrict expr to domain={\thisrow{channels}}{1:1},
restrict expr to domain={\thisrow{smoothness}}{10.0:10.0},
unbounded coords=discard]
table[x = parameters, y = psnr]
{test_errors};

\addplot+[thick, dashdotted, mark options={solid}, blue, mark = diamond*,
restrict expr to domain={\thisrow{modelid}}{\DUFORTFRANKEL:\DUFORTFRANKEL},
restrict expr to domain={\thisrow{activationid}}{\PERONAMALIK:\PERONAMALIK},
restrict expr to domain={\thisrow{couplingid}}{\UNCOUPLED:\UNCOUPLED},
restrict expr to domain={\thisrow{channels}}{1:1},
restrict expr to domain={\thisrow{smoothness}}{10.0:10.0},
unbounded coords=discard]
table[x = parameters, y = psnr]
{test_errors};

\addplot+[thick, densely dotted, mark options={solid}, black!60!green, mark = 
star,
restrict expr to domain={\thisrow{modelid}}{\FSI:\FSI},
restrict expr to domain={\thisrow{activationid}}{\PERONAMALIK:\PERONAMALIK},
restrict expr to domain={\thisrow{couplingid}}{\UNCOUPLED:\UNCOUPLED},
restrict expr to domain={\thisrow{channels}}{1:1},
restrict expr to domain={\thisrow{smoothness}}{10.0:10.0},
unbounded coords=discard]
table[x = parameters, y = psnr]
{test_errors};

\addplot+[dashed, black, thick, mark=none, domain=0.1:50000] plot 
(\x,\bestPMpsnr);
\end{axis}

\end{tikzpicture}

%% file: exp_channels_relu.tex
\begin{tikzpicture}[scale = 1.0]

\begin{axis}
[
width=0.9*\axisdefaultwidth,
height=0.75*\axisdefaultwidth,
samples=500,
axis lines=left,
ymin = 30.5, ymax = 42,
xmin = 700, xmax = 40000, xmode=log,
ytick = {31, 33, ..., 41},
ylabel={PSNR},
xlabel={number of parameters, logarithmic scale},
xlabel near ticks,
ylabel near ticks,
grid = major,
legend entries={ResNet,
                SymResNet, 
                Du Fort--Frankel,
                FSI,
                Linear diffusion},
title = {\textbf{ReLU activation}},
legend cell align=left,
legend style={at={(0.41, 0.32)}, anchor = west}
]

\addplot+[thick, dotted, mark options={solid}, black, mark = *,
restrict expr to domain={\thisrow{modelid}}{\STANDARDRESNET:\STANDARDRESNET},
restrict expr to domain={\thisrow{activationid}}{\RELU:\RELU},
restrict expr to domain={\thisrow{couplingid}}{\UNCOUPLED:\UNCOUPLED},
restrict expr to domain={\thisrow{channels}}{16:16},
restrict expr to domain={\thisrow{kernelsize}}{3:3},
restrict expr to domain={\thisrow{smoothness}}{0.0:0.0},
unbounded coords=discard]
table[x = parameters, y = psnr]
{test_errors};

\addplot+[thick, dashed, mark options={solid}, red, mark = square*,
restrict expr to domain={\thisrow{modelid}}{\STABLERESNET:\STABLERESNET},
restrict expr to domain={\thisrow{activationid}}{\RELU:\RELU},
restrict expr to domain={\thisrow{couplingid}}{\UNCOUPLED:\UNCOUPLED},
restrict expr to domain={\thisrow{channels}}{16:16},
restrict expr to domain={\thisrow{kernelsize}}{3:3},
restrict expr to domain={\thisrow{smoothness}}{0.0:0.0},
unbounded coords=discard]
table[x = parameters, y = psnr]
{test_errors};

\addplot+[thick, dashed, mark options={solid}, blue, mark = diamond*,
restrict expr to domain={\thisrow{modelid}}{\DUFORTFRANKEL:\DUFORTFRANKEL},
restrict expr to domain={\thisrow{activationid}}{\RELU:\RELU},
restrict expr to domain={\thisrow{couplingid}}{\UNCOUPLED:\UNCOUPLED},
restrict expr to domain={\thisrow{channels}}{16:16},
restrict expr to domain={\thisrow{kernelsize}}{3:3},
restrict expr to domain={\thisrow{smoothness}}{0.0:0.0},
unbounded coords=discard]
table[x = parameters, y = psnr]
{test_errors};

\addplot+[thick, densely dotted, mark options={solid}, black!60!green, mark = 
star,
restrict expr to domain={\thisrow{modelid}}{\FSI:\FSI},
restrict expr to domain={\thisrow{activationid}}{\RELU:\RELU},
restrict expr to domain={\thisrow{couplingid}}{\UNCOUPLED:\UNCOUPLED},
restrict expr to domain={\thisrow{channels}}{16:16},
restrict expr to domain={\thisrow{kernelsize}}{3:3},
restrict expr to domain={\thisrow{smoothness}}{0.0:0.0},
unbounded coords=discard]
table[x = parameters, y = psnr]
{test_errors};

\addplot+[dashed, black, thick, mark=none, domain=0.1:50000] plot 
(\x,\bestHDpsnr);
\end{axis}

\end{tikzpicture}

%% file: exp_channels_charb.tex
\begin{tikzpicture}[scale = 1.0]

\begin{axis}
[
width=0.9*\axisdefaultwidth,
height=0.75*\axisdefaultwidth,
samples=500,
axis lines=left,
ymin = 30.5, ymax = 42,
xmin = 700, xmax = 40000, xmode=log,
ytick = {31, 33, ..., 41},
ylabel={PSNR},
xlabel={number of parameters, logarithmic scale},
xlabel near ticks,
ylabel near ticks,
grid = major,
legend entries={ResNet,
                SymResNet, 
                Du Fort--Frankel,
                FSI,
                Charbonnier diffusion},
title = {\textbf{Charbonnier activation}},
legend cell align=left,
legend style={at={(0.3, 0.25)}, anchor = west}
]

\addplot+[thick, dotted, mark options={solid}, black, mark = *,
restrict expr to domain={\thisrow{modelid}}{\STANDARDRESNET:\STANDARDRESNET},
restrict expr to domain={\thisrow{activationid}}{\CHARBONNIER:\CHARBONNIER},
restrict expr to domain={\thisrow{couplingid}}{\UNCOUPLED:\UNCOUPLED},
restrict expr to domain={\thisrow{channels}}{16:16},
restrict expr to domain={\thisrow{kernelsize}}{3:3},
restrict expr to domain={\thisrow{smoothness}}{0.0:0.0},
unbounded coords=discard]
table[x = parameters, y = psnr]
{test_errors};

\addplot+[thick, dashed, mark options={solid}, red, mark = square*,
restrict expr to domain={\thisrow{modelid}}{\STABLERESNET:\STABLERESNET},
restrict expr to domain={\thisrow{activationid}}{\CHARBONNIER:\CHARBONNIER},
restrict expr to domain={\thisrow{couplingid}}{\UNCOUPLED:\UNCOUPLED},
restrict expr to domain={\thisrow{channels}}{16:16},
restrict expr to domain={\thisrow{kernelsize}}{3:3},
restrict expr to domain={\thisrow{smoothness}}{0.0:0.0},
unbounded coords=discard]
table[x = parameters, y = psnr]
{test_errors};

\addplot+[thick, dashdotted, mark options={solid}, blue, mark = diamond*,
restrict expr to domain={\thisrow{modelid}}{\DUFORTFRANKEL:\DUFORTFRANKEL},
restrict expr to domain={\thisrow{activationid}}{\CHARBONNIER:\CHARBONNIER},
restrict expr to domain={\thisrow{couplingid}}{\UNCOUPLED:\UNCOUPLED},
restrict expr to domain={\thisrow{channels}}{16:16},
restrict expr to domain={\thisrow{kernelsize}}{3:3},
restrict expr to domain={\thisrow{smoothness}}{0.0:0.0},
unbounded coords=discard]
table[x = parameters, y = psnr]
{test_errors};

\addplot+[thick, densely dotted, mark options={solid}, black!60!green, mark = 
star,
restrict expr to domain={\thisrow{modelid}}{\FSI:\FSI},
restrict expr to domain={\thisrow{activationid}}{\CHARBONNIER:\CHARBONNIER},
restrict expr to domain={\thisrow{couplingid}}{\UNCOUPLED:\UNCOUPLED},
restrict expr to domain={\thisrow{channels}}{16:16},
restrict expr to domain={\thisrow{kernelsize}}{3:3},
restrict expr to domain={\thisrow{smoothness}}{0.0:0.0},
unbounded coords=discard]
table[x = parameters, y = psnr]
{test_errors};

\addplot+[dashed, black, thick, mark=none, domain=0.1:50000] plot 
(\x,\bestCHpsnr);
\end{axis}

\end{tikzpicture}

%% file: exp_channels_pm.tex
\begin{tikzpicture}[scale = 1.0]

\begin{axis}
[
width=0.9*\axisdefaultwidth,
height=0.75*\axisdefaultwidth,
samples=500,
axis lines=left,
ymin = 30.5, ymax = 42,
xmin = 700, xmax = 40000, xmode=log,
ytick = {31, 33, ..., 41},
ylabel={PSNR},
xlabel={number of parameters, logarithmic scale},
xlabel near ticks,
ylabel near ticks,
grid = major,
legend entries={ResNet,
                SymResNet, 
                Du Fort--Frankel,
                FSI,
                Perona--Malik diffusion},
title = {\textbf{Perona--Malik activation}},
legend cell align=left,
legend style={at={(0.3, 0.25)}, anchor = west}
]

\addplot+[thick, dotted, mark options={solid}, black, mark = *,
restrict expr to domain={\thisrow{modelid}}{\STANDARDRESNET:\STANDARDRESNET},
restrict expr to domain={\thisrow{activationid}}{\PERONAMALIK:\PERONAMALIK},
restrict expr to domain={\thisrow{couplingid}}{\UNCOUPLED:\UNCOUPLED},
restrict expr to domain={\thisrow{channels}}{16:16},
restrict expr to domain={\thisrow{kernelsize}}{3:3},
restrict expr to domain={\thisrow{smoothness}}{0.0:0.0},
unbounded coords=discard]
table[x = parameters, y = psnr]
{test_errors};

\addplot+[thick, dashed, mark options={solid}, red, mark = square*,
restrict expr to domain={\thisrow{modelid}}{\STABLERESNET:\STABLERESNET},
restrict expr to domain={\thisrow{activationid}}{\PERONAMALIK:\PERONAMALIK},
restrict expr to domain={\thisrow{couplingid}}{\UNCOUPLED:\UNCOUPLED},
restrict expr to domain={\thisrow{channels}}{16:16},
restrict expr to domain={\thisrow{kernelsize}}{3:3},
restrict expr to domain={\thisrow{smoothness}}{0.0:0.0},
unbounded coords=discard]
table[x = parameters, y = psnr]
{test_errors};

\addplot+[thick, dashdotted, mark options={solid}, blue, mark = diamond*,
restrict expr to domain={\thisrow{modelid}}{\DUFORTFRANKEL:\DUFORTFRANKEL},
restrict expr to domain={\thisrow{activationid}}{\PERONAMALIK:\PERONAMALIK},
restrict expr to domain={\thisrow{couplingid}}{\UNCOUPLED:\UNCOUPLED},
restrict expr to domain={\thisrow{channels}}{16:16},
restrict expr to domain={\thisrow{kernelsize}}{3:3},
restrict expr to domain={\thisrow{smoothness}}{0.0:0.0},
unbounded coords=discard]
table[x = parameters, y = psnr]
{test_errors};

\addplot+[thick, densely dotted, mark options={solid}, black!60!green, mark = 
star,
restrict expr to domain={\thisrow{modelid}}{\FSI:\FSI},
restrict expr to domain={\thisrow{activationid}}{\PERONAMALIK:\PERONAMALIK},
restrict expr to domain={\thisrow{couplingid}}{\UNCOUPLED:\UNCOUPLED},
restrict expr to domain={\thisrow{channels}}{16:16},
restrict expr to domain={\thisrow{kernelsize}}{3:3},
restrict expr to domain={\thisrow{smoothness}}{0.0:0.0},
unbounded coords=discard]
table[x = parameters, y = psnr]
{test_errors};

\addplot+[dashed, black, thick, mark=none, domain=0.1:50000] plot 
(\x,\bestPMpsnr);
\end{axis}

\end{tikzpicture}

%% file: fmg.tex
\begin{tikzpicture}

\draw[-, dashed, thick] (0,0) -- (7,0);
\draw[-, dashed, thick] (0,1) -- (7,1);
\draw[-, dashed, thick] (0,2) -- (7,2);

\node[yshift=-5mm, anchor=east] (coarse) at (0,0) {\textbf{coarse}};
\node[xshift=-3mm] (hhhh) at (0,0) {$4h$};
\node[xshift=-3mm] (hhhh) at (0,1) {$2h$};
\node[xshift=-3mm] (hhhh) at (0,2) {$h$};
\node[yshift=5mm, anchor=east] (fine) at (0,2) {\textbf{fine}};


\node[shape=circle, fill,scale=0.6] (A) at (0.3,0) {};
\node[shape=circle, fill,scale=0.6] (B) at (0.3+1*6.4/11,1) {};
\node[shape=circle, fill,scale=0.6] (C) at (0.3+2*6.4/11,0) {};
\node[shape=circle, fill,scale=0.6] (D) at (0.3+3*6.4/11,1) {};
\node[shape=circle, fill,scale=0.6] (E) at (0.3+4*6.4/11,2) {};
\node[shape=circle, fill,scale=0.6] (F) at (0.3+5*6.4/11,1) {};
\node[shape=circle, fill,scale=0.6] (G) at (0.3+6*6.4/11,0) {};
\node[shape=circle, fill,scale=0.6] (H) at (0.3+7*6.4/11,1) {};
\node[shape=circle, fill,scale=0.6] (I) at (0.3+8*6.4/11,0) {};
\node[shape=circle, fill,scale=0.6] (J) at (0.3+9*6.4/11,1) {};
\node[shape=circle, fill,scale=0.6] (K) at (0.3+10*6.4/11,2) {};

\draw[-, thick] (A) -- (B);
\draw[-, thick] (B) -- (C);
\draw[-, thick] (C) -- (D);
\draw[-, thick] (D) -- (E);
\draw[-, thick] (E) -- (F);
\draw[-, thick] (F) -- (G);
\draw[-, thick] (G) -- (H);
\draw[-, thick] (H) -- (I);
\draw[-, thick] (I) -- (J);
\draw[-, thick] (J) -- (K);

\end{tikzpicture}